\newif\ifpictures
\picturestrue

\documentclass[12pt]{amsart}
\usepackage{amssymb,amsmath}
 \usepackage{amsopn}
 \usepackage{xspace}
 \usepackage{hyperref}
 \usepackage[dvips]{graphicx}
 \usepackage[arrow,matrix,curve]{xy}
  \usepackage{setspace}
\usepackage {color, tikz}

\numberwithin{equation}{section}
\newtheorem{thm}{Theorem}
\newtheorem{prop}[thm]{Proposition}
\newtheorem{lemma}[thm]{Lemma}
\newtheorem{cor}[thm]{Corollary}

\numberwithin{thm}{section}

\theoremstyle{definition}

\newtheorem*{defn*}{Definition}
\newtheorem{exa}[thm]{Example}
\newtheorem*{exa*}{Example}
\newtheorem{rem}[thm]{Remark}

\newcounter{FNC}[page]
\def\newfootnote#1{{\addtocounter{FNC}{2}$^\fnsymbol{FNC}$%
     \let\thefootnote\relax\footnotetext{$^\fnsymbol{FNC}$#1}}}
     
\definecolor{DarkGreen}{rgb}{0,0.65,0}

\newcommand{\C}{\mathbb{C}}

\newcommand{\F}{\mathbb{F}}

\newcommand{\N}{\mathbb{N}}

\newcommand{\Q}{\mathbb{Q}}
\newcommand{\R}{\mathbb{R}}

\newcommand{\Z}{\mathbb{Z}}

\newcommand\cA{{\ensuremath{\mathcal{A}}}\xspace}
\newcommand\cB{{\ensuremath{\mathcal{B}}}\xspace}

\newcommand\cV{{\ensuremath{\mathcal{V}}}\xspace}

\newcommand{\eps}{\varepsilon}

\newcommand{\alp}{\alpha}
\newcommand{\lam}{\lambda}

\newcommand{\lf}{\left}
\newcommand{\ri}{\right}
\newcommand{\ra}{\rightarrow}

\newcommand{\Lera}{\Leftrightarrow}
\newcommand{\ovl}{\overline}
\newcommand{\wh}{\widehat}

\DeclareMathOperator{\conv}{conv}

\DeclareMathOperator{\Arg}{Arg}
\DeclareMathOperator{\Log}{Log}
\DeclareMathOperator{\New}{New}
\DeclareMathOperator{\resultant}{Res}

\DeclareMathOperator{\odd}{odd} 
\DeclareMathOperator{\even}{even} 
\DeclareMathOperator{\ord}{ord}

\DeclareMathOperator{\RE}{Re}
\DeclareMathOperator{\IM}{Im}

\DeclareMathOperator{\coA}{\text{co}\cA}

\title{Norms of Roots of Trinomials}

\author{Thorsten Theobald}

\author{Timo de Wolff}

\address{Thorsten Theobald, Goethe-Universit\"at, FB 12 -- Institut f\"ur Mathematik,
Postfach 11 19 32, D-60054 Frankfurt am Main, Germany \medskip \newline 
\hspace*{8pt} Timo de Wolff, Texas A\&M University, Department of Mathematics, College Station, TX 77843-3386, 
 USA\medskip
}

\email{theobald@math.uni-frankfurt.de, dewolff@math.tamu.edu}


\begin{document}

\begin{abstract}
The behavior of norms of roots of univariate trinomials $z^{s+t} + p z^t + q \in \mathbb{C}[z]$ for fixed support $A = \{0,t,s+t\} \subset \mathbb{N}$ with respect to the choice of
coefficients $p,q \in \mathbb{C}$ is a classical late 19th and early 20th century problem. 
Although algebraically characterized by P.\ Bohl in 1908, the geometry and topology of the corresponding
parameter space of coefficients had yet to be revealed. 
Assuming $s$ and $t$ to be coprime we provide such a characterization for the space of trinomials
by reinterpreting the problem in terms of amoeba theory.
The roots of given norm are parameterized in terms of a hypotrochoid curve
along a $\C$-slice of the space of trinomials, with multiple roots of this norm
appearing exactly on the singularities. 
As a main result, we show that the set of all trinomials with support $A$ and
certain roots of identical norm, as well as its complement can be deformation 
retracted to the torus knot $K(s+t,s)$, and thus are connected but not simply connected. 
An exception is the case where the $t$-th smallest norm coincides with 
the $(t+1)$-st smallest norm. 
Here, the complement has a different topology since
it has fundamental group $\Z^2$. 
%
%\bigskip
%
%\bigskip
%
%\bigskip
%
%\noindent T.T.: Goethe-Universit\"at, FB 12 -- Institut f\"ur Mathematik,
%Postfach 11 19 32, D-60054 Frankfurt am Main, Germany,
%{\sf theobald@math.uni-frankfurt.de}
%
%\bigskip
%
%\noindent T.d.W.: Texas A\&M University, Department of Mathematics, College Station, TX 77843-3386, USA,
%{\sf dewolff@math.tamu.edu}
\end{abstract}

\maketitle

\section{Introduction}
\label{Sec:Introduction}

The investigation of \textit{univariate trinomials}, i.e., polynomials of the form
\begin{eqnarray}
	z^{s+t} + p z^t + q \in \C[z] \label{Equ:Trinomial}
        \text{ with } s,t \in \N^*
\end{eqnarray}
is a truly classical late nineteenth and early twentieth century problem (see, e.g.,
\cite{Biernaky,Bohl,Egervary3,Farkas2,Kemper,Landau,Nekrassoff}). At this time mathematicians started to ask how the $s+t$ complex roots depend on the choice of the coefficients $p,q$. For example, how the roots can be characterized geometrically, how many of them lie in a disk of given radius or whether two roots share the same norm.

Algebraically, these questions are well understood -- particularly due to P.~Bohl's results
from 1908 (\cite{Bohl}; stated in 
Theorems \ref{Thm:Bohl1}, \ref{Thm:Bohl2} below). And also the geometry of roots
in the complex plane is well described by P.~Nekrassoff in 1887 \cite{Nekrassoff} and  J.~Egerv\'{a}ry in 1922--1931 (see the survey~\cite{Szabo}). But, after more than a century has passed and although the investigation of trinomials went on in modern times (e.g., \cite{Dilcher:Nulton:Stolarsky,Fell,Melman}), the parameter space of coefficients and in particular its geometric and topological properties have still not been understood.

Let $T_A$ denote the space of all trinomials with support set $A = \{0,t,s+t\}$ such that $s$ and $t$ are coprime. Since we
usually assume $p,q \neq 0$, $T_A$ can be identified with the two-dimensional space of 
parameters $(p,q) \in (\C^*)^2$. Immediate first questions on the space of trinomials are:
\begin{enumerate}
	\item[(A)] What is, for given $q$, the geometric structure of the set of all $p$ such that $f$ has a 
  root with norm $v$?
	\item[(B)] What is, for given $q$, the geometric structure of the set of all $p$ such that $f$ has two roots of norm $v$, respectively of the same norm at all?
\end{enumerate}
Specifically, we aim at semialgebraic and parametric descriptions of these sets.

Denote by $U_j^A$ the subset of trinomials in $T_A$ whose $j$-th and $(j+1)$-th smallest root (ordered by their norm)
have distinct norm, $1 \le j \le s+t-1$. Formally, we also consider $U^A_0$ and $U^A_{s+t}$, by 
declaring $f \in U^A_0$ and $f \in U_{s+t}^A$ for every trinomial $f \in T_A$. 
For a given $f \in T_A$ a classical question is to determine the subset $J \subseteq \{0, \ldots, s+t\}$
such that $f \in U_j^A$ if and only if $j \in J$. More globally, we ask:

\begin{enumerate}
\item[(C)] Which geometric and topological properties do the sets $U_j^A \subseteq T_A$ and 
       their complements have?
\end{enumerate}

In this article we reinterpret the classical problems about the norms of roots of trinomials in terms of \textit{amoeba theory} and show that this tool-set allows to solve these problems and to uncover a beautiful geometric and topological structure hidden in the parameter space of trinomials. 

For a given Laurent polynomial $f \in \C[z_1^{\pm 1},\ldots,z_n^{\pm 1}]$ with zero set $\cV(f) \subseteq (\C^*)^n$ the \textit{amoeba} $\cA(f)$ (introduced by Gelfand, Kapranov and Zelevinsky in \cite{Gelfand:Kapranov:Zelevinsky}) is the image of $\cV(f)$ under the $\log$-absolute-value map
\begin{eqnarray}
\label{Equ:Amoebadef}
	\Log|\cdot| : \lf(\C^*\ri)^n \ra \R^n, \quad \ (z_1,\ldots,z_n) \mapsto (\log|z_1|, \ldots, \log|z_n|) \, .
\end{eqnarray}
Amoebas not only have strong structural properties and connections to various fields of 
mathematics including complex analysis (e.g., \cite{Forsberg:Passare:Tsikh,Passare:Rullgard:Spine}), topology 
of real curves (e.g., \cite{Mikhalkin:Annals}) and tropical geometry (e.g., \cite{Maclagan:Sturmfels,Maslov,Mikhalkin:Enumerative}), but also turn out to be a canonical and powerful tool to understand the connection between varieties and parameter spaces of polynomials (see, e.g., \cite{Gelfand:Kapranov:Zelevinsky,Rullgard:Franzoesisch,Rullgard:Diss,Theobald:deWolff:Genus1}).

\smallskip

With regard to the trinomial setup, our point of departure is that
a trinomial~\eqref{Equ:Trinomial} 
with $q \neq 0$ has a root of a given norm $v \in \R_{> 0}$ if and only if 
$p$ is located on an algebraic \textit{hypotrochoid curve} depending on $q$, the exponents $s,t$ and, of course, $v$ itself (Theorem \ref{Thm:Hypotrochoid}). Hypotrochoids are well-known special instances of roulette curves in the complex plane, see, e.g., \cite{Brieskorn:Knoerrer,Fladt} for many of their nice
properties.

We show that two roots share the same norm $v$ if and only if the coefficient $p$ is located on a singularity (in general, a node) of the particular hypotrochoid (Theorem \ref{Thm:MultipleNormsCorrespondToSingularities}). Moreover,
there exist two roots with the same norm if and only if $p$ is located on a particular union of $2(s+t)$ rays $F(s,t,q)$ in the corresponding $\C$-slice of the parameter space. $F(s,t,q)$ is thus determined by the support set and $q$ (Theorems \ref{Thm:ComplementIn1Fan} and \ref{Thm:LocalStructureUalpha}). 

By additionally studying the discriminants of trinomials, we provide a complete answer to
question (B) through Theorems~\ref{Thm:LocalStructureUalpha} and Corollary~\ref{Cor:LocalStructureUzero} below,
where it is somewhat unexpected that there are differences between the characterizations of
the sets $U_j^A$ for $j \in \{1, \ldots, s+t-1\} \setminus \{j\}$ and for $j=t$.
Furthermore, this allows one to prove that only particular roots (with respect to the ordering induced by the norm) can have multiplicity two (Corollary \ref{Cor:Discriminant}). Moreover, only particular roots of real trinomials can be real (Theorem \ref{Thm:LocationRealRoots}). Geometrically, in the case of roots with multiplicity two the hypotrochoid deforms to a hypocycloid and $p$ is located on a cusp instead of a node. 

For the variant of problem (A) in which the coefficient $p$ instead of $q$ is fixed, we obtain similar 
results involving epitrochoids instead of hypotrochoids (Theorem \ref{Thm:Epitrochoid}).

This local description of the parameter space then allows us to tackle Problem (C) and 
reveal the topology of the parameter space of all trinomials in $T_A$.
We show that for all $j \in\{1,\ldots,s+t-1\} \setminus \{t\}$ the set $U_j^A$ as well as its complement $(U_j^A)^c = T_A \setminus U_j^A$ is a connected but not simply connected set. Namely, both $U_j^A$ and $(U_j^A)^c$ can be deformation retracted to the torus knot $K(s+t,s)$ (Theorem \ref{Thm:TopologyTrinomials}), which is a closed path on a standard torus. Hence, the fundamental group of these sets $U_j^A$ and $(U_j^A)^c$ is $\Z$. The same holds for $(U_t^A)^c$ and the zero set $\cV(D)$ of the discriminant $D$ of trinomials, where furthermore $\cV(D)$ is a deformation retract of $(U_t^A)^c$. $U_t^A$ is also connected and not simply connected, but its topology is different since it has fundamental group $\Z^2$ (Theorem \ref{Thm:TopologyMiddleTerm}). Note that complements are taken in $T_A \cong (\C^*)^2$ and therefore the fundamental groups of these complements can differ from fundamental groups of $\R^3 \setminus K(s+t,s)$.

The article is organized as follows. In Section \ref{Sec:Preliminaries}, we fix our notation and introduce some facts from amoeba theory and about fibrations. In Section \ref{Sec:TrinomialsClassical} we 
review the classical questions and results on trinomials developed mostly during 1880-1930, 
as well as some modern facts. Section~\ref{Sec:TrinomialsModulis} deals with the local structure of the parameter space along $\C$-slices given by fixing one of the two coefficients. 
In Section \ref{Sec:TrinomialsTopology} we investigate the complete parameter space and provide the topological description of the sets $U_j^A$, their complements and the zero set $\cV(D)$ of
the discriminant. Section~\ref{Sec:Openquestions} closes the paper with some final remarks
on the (widely open) extension of our trinomial results to the case of polynomials with
general support set.

We remark that parts of the results of this article are contained in the thesis \cite{deWolff:Diss} of the second author.

\subsection*{Acknowledgements}

We thank Jens Forsg{\aa}rd and Maurice Rojas for helpful comments and for bringing various additional aspects to our attention. We are also grateful to an anonymous referee for detailed suggestions.

The first author was partially supported by DFG projects TH 1333/2-1 and 1333/3-1. The second author was partially supported by  DFG project TH 1333/2-1, GIF Grant no. 1174/2011 and DFG project MA 4797/3-2.

\section{Preliminaries}
\label{Sec:Preliminaries}
\subsection{Amoebas}
\label{SubSec:Amoebas}
We collect some facts and notation from amoeba theory and afterwards restrict ourselves to the univariate case. For further information, next to the
fundamental reference \cite{Gelfand:Kapranov:Zelevinsky}, see, e.g., \cite{deWolff:Diss,Mikhalkin:Survey,Passare:Tsikh:Survey,Rullgard:Diss}.

For a multivariate polynomial $f = \sum_{\alpha \in A} b_\alpha \mathbf{z}^{\alpha} \in \C[\mathbf{z}^{\pm 1}]$ over a finite support set $A \subseteq \Z^n$ the amoeba $\cA(f) \subseteq \R^n$ as defined in~\eqref{Equ:Amoebadef} is a closed set with non-empty complement and each component of the complement of $\cA(f)$ is convex (see \cite{Gelfand:Kapranov:Zelevinsky}). Furthermore, every component of the complement of a given amoeba $\cA(f)$ corresponds to a unique lattice point in the Newton polytope $\New(f)$ of $f$ via the \emph{order map} (see \cite{Forsberg:Passare:Tsikh}),
\begin{eqnarray}
  \label{Equ:Ordermap}
\ord: \R^n \setminus \cA(f) & \to & \New(f) \cap \Z^n, \quad \mathbf{w} \mapsto (u_1,\ldots,u_n) \ \text{ with } \\
  u_j & = & \frac{1}{(2\pi i)^n} \int_{\Log|\mathbf{z}| = \mathbf{w}} \frac{z_j \partial_j f(\mathbf{z})}{f(\mathbf{z})}
    \frac{dz_1 \cdots dz_n}{z_1 \cdots z_n} \, , \quad 1 \le j \le n \, .\nonumber
\end{eqnarray}
Notice that this map indeed is constant on each component of the complement of $\cA(f)$. As a consequence, we define for each $\alp \in \New(f) \cap \Z^n$ the set
\begin{eqnarray*}
	E_{\alpha}(f) & = & \{\mathbf{w} \in \R^n \setminus \cA(f) \ : \ \ord(\mathbf{w}) = \alpha\},
\end{eqnarray*}
i.e., the set of all points in the complement of the amoeba $\cA(f)$, which have order $\alpha$.

For a fixed support set $A \subseteq \Z^n$, we can identify every polynomial with its coefficient
vector. Thus, we can identify the parameter space $(\C^*)^A$ of polynomials with support set $A$ with a $(\C^*)^d$, where $d = \# A$. One key problem in amoeba theory is to understand the sets
\begin{eqnarray*}
	U_\alp^A & = & \{f \in (\C^*)^A \ : \ E_{\alpha}(f) \neq \emptyset\},
\end{eqnarray*}
i.e., the set of all polynomials with Newton polytope $A$, whose amoebas have a component in the complement of order $\alpha$ (see, e.g., \cite[Remark 1.10, p.\ 198]{Gelfand:Kapranov:Zelevinsky}).
 These sets were systematically studied first by Rullg{\aa}rd and turn out to have nice structural properties. E.g., they are open, semi-algebraic sets, which are non-empty for all $\alpha \in A$ (see \cite{Rullgard:Franzoesisch,Rullgard:Diss}).

\smallskip

We describe the \textit{lopsidedness} condition introduced by Purbhoo in \cite{Purbhoo} and similarly used before by Passare, Rullg{\aa}rd et.\ al.\ \cite{Forsberg:Passare:Tsikh,Rullgard:Diss}. For a given $f = \sum_{j = 1}^d b_j \mathbf{z}^{\alpha(j)} \in \C[z_1^{\pm 1},\ldots,z_n^{\pm 1}]$ and $\mathbf{v} \in \R_{>0}^n$ we say that $f$ is \textit{lopsided} at $\Log|\mathbf{v}|$ if 
one of the entries in the list
\begin{eqnarray*}
	f\{\mathbf{v}\} & = & \left[ |b_1 \mathbf{v}^{\alpha(1)}|,\ldots,|b_d \mathbf{v}^{\alpha(d)}| \right]
\end{eqnarray*}
is larger than the sum of all of the others. Clearly, if $f$ is lopsided at $\Log|\mathbf{v}|$ then $\Log|\mathbf{v}| \notin \cA(f)$. Furthermore, if $|b_j \mathbf{v}^{\alpha(j)}|$ is the dominating term in the lopsided list $f\{\mathbf{v}\}$, then $\ord(\Log|\mathbf{v}|) = \alpha(j)$ (see \cite[Proposition 2.7]{Forsberg:Passare:Tsikh} and \cite[Proposition 4.1]{Purbhoo}).

\smallskip

In this article we investigate complex univariate trinomials $f = z^{s+t} + p z^t + q \in \C[z]$
with $s,t \in \N^*$ (i.e., $A = \{0,s,s+t\} \subseteq \N$). Mostly, we assume $q \in \C^*$. 
Furthermore, we always assume that $s,t$ are coprime, because all other cases can be traced back to those instances via the substitution $z^{\gcd(s,t)} \mapsto z$. For univariate polynomials, most objects
from amoeba theory are represented by well-known (classical) objects and theorems, as explained in the following. This is convenient, since it allows us to argue in, say, classical terms and let the amoeba machinery run in the background.

Let us assume that $f = (z - a_1) \cdots (z - a_{s+t})$,
where multiple roots of $f$ are allowed.
We can always assume  $|a_1| \leq \cdots \leq |a_{s+t}|$.
Hence, the amoeba $\cA(f)$ is the set of $s+t$ points $\log|a_1|,\ldots,\log|a_{s+t}|$ on the real line. In the case of univariate polynomials the order map for amoebas coincides with the classical argument principle from complex analysis (in fact, the order map \textit{is} nothing else than an extension of the argument principle to the multivariate case; see \cite{Forsberg:Passare:Tsikh} for further details). Recall that for a univariate complex Laurent polynomial $f \in \C[z^{\pm 1}]$ and a region $R \subseteq \C$ such that $\partial R$ is a closed curve satisfying $\partial R \cap \cV(f) = \emptyset$ the argument principle (see, e.g., \cite{Gamelin}) states that 
\begin{eqnarray*}
	\frac{1}{2\pi i} \int_{\partial R} \frac{f'(z)}{f(z)} dz & = & \#\text{ roots } - \# \text{ poles inside } R.
\end{eqnarray*}

Since a trinomial $f$ of the form \eqref{Equ:Trinomial} has no poles, for every $w \in \R \setminus \cA(f)$ we have $\ord(w) = k$ with $k = \max\{ 0 \leq j \leq s+t \ : \ \log|a_j| < w\}$, i.e., the number of roots of $f$ with norm smaller than $w$. Since we defined $E_{\alpha}(f)$ as the set of points in the amoeba complement with order $\alpha$, the univariate situation specializes to
\begin{eqnarray}
	E_{j}(f) & = & \{ w \in \R : \log|a_j| < w < \log|a_{j+1}| \} \ \text{ for } \ 0 \leq j \leq s+t.
\label{eq:component-univariate}
\end{eqnarray}

For trinomials with support set $A$, and restricting to the case that the coefficient $p$ is non-zero,
the sets $U_{\alpha}^A$ specialize to
\begin{eqnarray*}
	U_{j}^A & = & \{f \in T_A \cong (\C^*)^2 \ : \ |a_j| \neq |a_{j+1}|\} \, , \quad
      0 \le j \le s+t.
\end{eqnarray*}
If the context is clear, then we use the short notation $U_j$ instead of $U_j^A$. Further note
that with one exception in Section \ref{Sec:TrinomialsTopology}, which we point out explicitly, for our investigations of the sets $U_{j}^A$, it does not matter if we consider  $U_j^A$ as a subset of $(\C^*)^2$ or if we consider the slight extension allowing $p =0$. Since for $p=0$ all roots have the same norm $\sqrt[s+t]{|q|}$, we know
that in this situation all the sets $U_j^{A}$ are empty for $1 \le j \le s+t-1$.

In the univariate case, the lopsidedness condition coincides with Pellet's classical theorem (see \cite{Pellet}). Here, with respect to trinomials, we refer to Section \ref{Sec:TrinomialsClassical}, where we will see that it specializes to a classical result by Bohl (Theorem \ref{Thm:Bohl1}).

\subsection{Fibers}
\label{SubSec:Fibers}
It is a well-known fact that the $\Log|\cdot|$-map comes with a fiber bundle $(S^1)^n \to (\C^*)^n \to \R^n$ given by the homeomorphism
\begin{eqnarray*}
	\Log_\C : (\C^*)^n \to \R^n \times (S^1)^n, \quad (z_1,\ldots,z_n) \mapsto (\log_\C(z_1),\ldots,\log_\C (z_n))
\end{eqnarray*}
for some chosen local branch of the holomorphic logarithm $\log_\C: \C^* \to \C$, $z \mapsto \log|z| + i \arg(z)$ (see, e.g., \cite{Mikhalkin:Annals,Mikhalkin:Survey}; see also \cite{deWolff:Diss}). 
That is, the following diagram commutes.
\begin{eqnarray*}
\begin{xy}
	\xymatrix{
	(\C^*)^n \ar[rr]^{\Log_\C} \ar[rd]_{\Log|\cdot|} & & \R^n \times (S^1)^n \ar[ld]^{\RE} \\
	& \R^n &
	}
\end{xy}
\end{eqnarray*}

Since the fibration works component-wise, we restrict ourselves to the univariate case $n = 1$ (i.e., the fiber bundle given by $\log|\cdot|$). For a given point $w \in \R_{\ge 0}$ the fiber $\F_{w}$ is
\begin{eqnarray*}
	\F_{w} & = & \{z \in \C^* \ : \ \log|z| = w\},
\end{eqnarray*}
which is obviously homeomorphic to the complex unit circle. For this article, the key fact is that the fiber bundle induces a \textit{fiber function} $f^{v}$ for every polynomial $f = \sum_{j = -k}^d b_j  z^j \in \C[z^{\pm 1}]$ and $v \in \R_{> 0}$ given by
\begin{eqnarray*}
	  f^{v}: S^1 \to \C, \quad  \phi \mapsto f\left(e^{\log|v| + i \cdot \phi}\right) = \sum_{j = -k}^d b_j  |v|^j \cdot e^{i \cdot \phi j}.
\end{eqnarray*}

That is, $f^{v}$ is the pullback $(\iota_{v})^*(f)$ of $f$ under the homeomorphism
\begin{eqnarray*}
 \iota_{v}: S^1 \to \F_{\log|v|} \subseteq \C^*, \quad \phi \mapsto e^{\log|v| + i \cdot \phi}.
\end{eqnarray*}
The zero set $\cV(f^{v})$ satisfies
$\cV(f^{v}) \ = \ \cV((\iota_{v})^*(f)) \ = \ \cV(f) \cap \F_{\log|v|}$, and, in particular,
\begin{eqnarray}
	\log|v| \in \cA(f) & \text{ iff } & \cV(f^{v}) \neq \emptyset. \label{Equ:FiberFunction2}
\end{eqnarray}
For an overview on fiber bundles see, e.g., \cite{Hatcher,Shafarevich}.

In Section \ref{Sec:TrinomialsTopology} we also need the $\Arg$-map, the natural counterpart of the $\Log|\cdot|$-map, given by
\begin{eqnarray*}
	\Arg: (\C^*)^n \to (S^1)^n, \quad (z_1,\ldots,z_n) \mapsto (\arg(z_1),\ldots,\arg(z_n)).
\end{eqnarray*}
The key fact for us is that with the same argument as above the $\Arg$-map also yields a natural fiber bundle structure $\R^{n} \ra (\C^*)^n \ra (S^1)^n$, which can be regarded as the canonical counterpart of the fibration of the $\Log|\cdot|$-map, since the following diagram commutes
\begin{eqnarray}
& & \begin{xy}
	\xymatrix{
	(\C^*)^n \ar[rr]^{\Log_\C} \ar[rd]_{\Arg} & & \R^n \times (S^1)^n \ar[ld]^{\IM} \\
	& (S^1)^n &
	}
\end{xy}.  \label{Equ:FiberBundleArgMap}
\end{eqnarray}

\section{Classical problems, classical results and modern developments}
\label{Sec:TrinomialsClassical}
 
Since the late 19th century, the connection between the roots of trinomials (often, in particular their norms) and the choice of their coefficients was studied intensively. We compile these classical as well as some modern results. 

An initial result, which attracted people to trinomial equations, was given by Bring in 1786 \cite{Bring1} showing that every univariate quintic can be transformed into a trinomial normal form $z^5 + a z + b$ via a suitable affine transformation. This result was (independently) reproven and generalized by Jerrard in 1852 \cite{Jerrard}; the resulting normal form is known as \textit{Bring-Jerrard (quintic) form}. For additional information see, e.g., the survey \cite{Adamchik:Jeffrey}.

In 1832/33 Bolyai showed that for a trinomial of the form $z^s - z - a$ with $s \in \N_{> 1}$ and $a > 0$ the recursive sequence $(x_n)_{n \in \N}$ given by $x_0 = 0$, $x_n = \sqrt[s]{a + x_{n-1}}$ converges to one of the trinomial roots for $n \to \infty$ \cite{Bolyai}. This \textit{Bolyai algorithm} was extended by Farkas in 1881 to trinomials of the form $z^s - b z - a$ with $a > 0$ and $b \in \R^*$, where
the sequence is not always converging if $b < 0$ (see \cite{Farkas1}; also \cite{Farkas2}). See the survey \cite{Szabo} for further details.

The first article investigating the geometric properties of roots of trinomials is, to the best of our knowledge, the fundamental work \cite{Nekrassoff} by Nekrassoff from 1887. He describes how roots of trinomials $z^{s+t} + p z^t + q$ with $p,q \in \C^*$ are located in certain disjoint regions (``Contouren'') of the complex plane. In other words, he gives bounds for the norms (described by converging series) and arguments for the different roots of the trinomials in dependence of $s,t,p$ and $q$. Similar results were obtained by Kemper in 1922 in a more general article about complex roots \cite{Kemper}.

In 1907, Landau proved that the minimal norm of the roots of a trinomial of the form $z^s + p z + q$ is bounded from above by $2 |q / p|$ and hence in particular independent of $s$ \cite{Landau}. Furthermore, he proved a similar bound for the minimal norm of a root of a tetranomial. These results were generalized to arbitrary univariate polynomials by Fej\'{e}r one year later \cite{Fejer} and also by Biernaky in 1923, who gave an upper bound for the first $t$ roots of an arbitrary trinomial \cite{Biernaky}; see also \cite{Fell}. 

The inverse of this question, i.e., to determine the number of roots $k \in \N$ with norm lower than a given $v \in \R_{>0}$, can be answered with a result by Bohl from 1908, see \cite{Bohl}. Specifically, he showed the following two theorems.

\begin{thm}(Bohl 1908)
Let $f = z^{s+t} + p z^t + q$ a trinomial with $p,q \in \C$. Let $v \in \R_{> 0}$ and $k$ be
the number of roots with norm smaller than $v$. Then the following holds.
$$\begin{array}{rclcl}

	\text{If } |q|			& > & v^{s+t} + |p| \cdot v^{t}, & \text{ then } & k = 0. \\
	\text{If } v^{s+t}		& > & |q| + |p| \cdot v^{t}, & \text{ then } & k = s+t. \\
	\text{If } |p| \cdot v^{t}	& > & |q| + v^{s+t}, & \text{ then } & k = t. \\
\end{array}$$
\label{Thm:Bohl1}
\end{thm}

This first theorem was already known before, since it is a special instance of Pellet's Theorem 
(\cite{Pellet}; see also \cite{Marden:Book}), which is concerned with
arbitrary, univariate polynomials.

Note that, from the viewpoint of amoeba theory, this theorem is also obvious since it treats the situation 
that $f$ is lopsided at $v$ and $k = \ord(\log|v|)$, which coincides with the exponent of the dominating term of the list $f\{v\}$ (see Section \ref{SubSec:Amoebas}). In other words, Theorem~\ref{Thm:Bohl1} is exactly the classical representation of lopsidedness from amoeba theory \cite{Purbhoo} for the special case of univariate trinomials.

The interesting, nontrivial case is described in a second statement. If none of the upper inequalities in Theorem \ref{Thm:Bohl1} holds, then there exists a (possibly degenerate) triangle $\Delta$ with edges of lengths $v^{s+t}, |p| \cdot v^t$ and $|q|$. Let $\alp=\measuredangle(|p|\cdot v^t,|q|)$
and $\beta = \measuredangle(v^{s+t},|q|)$.

\begin{thm}(Bohl 1908)
Let the notation be as in Theorem \ref{Thm:Bohl1}. If there exists a triangle $\Delta$ 
with edge lengths $v^{s+t}, |p| \cdot v^t$ and $|q|$, then the number $k$ of roots with norm smaller than $v \in \R_{>0}$ is given by the number of integers located in the open interval with endpoints
\begin{eqnarray}
\frac{(s+t)(\pi + \arg(p) - \arg(q)) - t(\pi -\arg(q))}{2\pi} - \frac{(s+t)\alp + t\beta}{2\pi} \label{Equ:Bohl1}
\end{eqnarray}
and 
\begin{eqnarray}
\frac{(s+t)(\pi + \arg(p) - \arg(q)) - t(\pi -\arg(q))}{2\pi} + \frac{(s+t)\alp + t\beta}{2\pi}. \label{Equ:Bohl2}	
\end{eqnarray}
\label{Thm:Bohl2}
\end{thm}

To illustrate the theorem, we give an example.

\begin{exa}
Let $f = z^3 + z + \sqrt{2}$ and $v = 1$. Then $\alp = \beta = \pi / 4$. Thus, $k$ is the number of integers between
\begin{eqnarray*}
\frac{3(\pi + 0 - 0 - \pi)}{2\pi} - \frac{3/4 \pi + 1/4 \pi}{2\pi}	& =	& -\frac{1}{2} \quad \text{ and }\\
\frac{3(\pi + 0 - 0 - \pi)}{2\pi}  + \frac{3/4 \pi + 1/4 \pi}{2\pi}	& =	& +\frac{1}{2}.
\end{eqnarray*}
Since this is only the origin, we have $k = 1$. A double check with \textsc{Maple} yields that the roots of $f$ have approximately norm 
$$0.83403883, 1.30216004 \text{ and } 1.30216004.$$
\end{exa}

Unfortunately, these theorems give, using the notation from Section \ref{SubSec:Amoebas}, no explanation for the geometric or topological structure of the parameter space $T_A$ or the sets $U_j^A$ in it. Amazingly, despite the fact that these theorems were proven over a century ago and people kept on investigating trinomials until nowadays (see below), no evident progress was made with respect to this geometric and topological structure. This fact will be the initial point for our own investigation.

Prior to this, we recall some fundamental results by Egerv\'{a}ry \cite{Egervary1,Egervary2,Egervary3,Egervary4} from 1922--1931 about trinomials. Again, we refer to the survey \cite{Szabo} by Szab\'{o}, where these classical results (partially written in Hungarian in the original) are presented in modern terminology. Egerv\'{a}ry calls two trinomials $f_1,f_2$ with coefficients $p_1,q_1$ and $p_2,q_2 \in \C^*$ and (both) with exponents $s,t$ coprime \textit{equivalent} if and only if
\begin{eqnarray*}
	f_1(z) \ = \ f_2(z \cdot e^{i \psi}) \ \text{ or } \ f_1(z) \ = \ \ovl{f_2}(z \cdot e^{i \psi})
\end{eqnarray*}
for some $\psi \in \R$.

\begin{thm}(Egerv\'{a}ry 1930)
For trinomials $f_1,f_2$ given as above, the following holds:
\begin{enumerate}
	\item $f_1$ and $f_2$ are equivalent if and only if
	\begin{eqnarray*}
		-(s+t)(\arg(p_1) \pm \arg(p_2)) + s(\arg(q_1) \pm \arg(q_2)) & \equiv & 0 \mod 2\pi.
	\end{eqnarray*}
	\item If $|p_1| = |p_2|$ and $|q_1| = |q_2|$, then the roots of $f_1$ and $f_2$ have the same
      norms if and only if $f_1$ and $f_2$ are equivalent. 
      Further, a trinomial has two roots with the same norm if and only if it is equivalent to a real trinomial.
	\item A trinomial has a root of multiplicity larger than one if and only if the coefficients $p,q$ of its
      equivalent real trinomial satisfy $(-1)^{s+t} q^s (s+t)^{s+t} = p^{s+t} s^s t^t$.
\end{enumerate}
\label{Thm:Egervary1}
\end{thm}

Egerv\'{a}ry showed not only algebraic properties of the roots of trinomials, but also gave a beautiful geometric description of their location in the complex plane, which explain why roots are located in the sections, which were described by Nekrassoff. For a trinomial $f$ of the 
form \eqref{Equ:Trinomial}, we define two polytopes in the complex plane
\begin{eqnarray*}
	P_s     & = & \conv\left\{\sqrt[s]{\frac{(2s + t)|p|}{s+t}} \cdot e^{i \frac{\arg(p) + (2j + 1)\pi}{s}} \ : \ 1 \leq j \leq s \right\}, \\
	P_{s+t} & = & \conv\left\{\sqrt[s+t]{\frac{(2s + t)|q|}{s}} \cdot e^{i \frac{\arg(q) + (2j + 1)\pi}{s+t}} \ : \ 1 \leq j \leq s+t \right\}.
\end{eqnarray*}

\begin{thm}(Egerv\'{a}ry 1930)
Let $f$ be of the form \eqref{Equ:Trinomial}. Then the roots of $f$ are exactly the equilibrium points of the force field of the unit masses at the vertices of $P_s$ and $P_{s+t}$.
\label{Thm:Egervary2}
\end{thm}

The investigation of trinomials went on in modern times. In 1980 Fell gave a geometric description of trajectories of roots of real trinomials in the complex plane under changing their coefficients \cite{Fell}. In 1992 Dilcher, Nulton and Stolarsky study (among other things) the zero distribution of the special trinomial $t z^{s+t} - (s+t) z^t + s$ with $s,t \in \N^*$ \cite{Dilcher:Nulton:Stolarsky}. And recently, in 2012, Melman improved Nekrassoffs results in \cite{Melman}.
\smallskip

\section{The Local Structure of the Parameter Space of Trinomials}
\label{Sec:TrinomialsModulis}

Given $A = \{0,s,s+t\}$, our goal is to describe the space $T_A$ of trinomials. More precisely, we 
determine the geometry of all trinomials with a root of a certain norm as well as the geometry and topology of the sets $U_j^A$ in $T_A$, as defined in the Introduction.

First we investigate the special case of a fixed $q \in \C^*$. In other words, we study $\C$-slices
\[
  (T_A)_q \ = \ \{ f = x^{s+t} + p x^t + q \, : \, p \in \C \} \ \cong \ \C
\]
 of $T_A$ (or $\C^*$-slices in case of assuming $p \neq 0$) 
and solve the initial questions locally along this slice. This allows us to provide two key results 
answering Problems (A) and (B).

We first observe that $f$ has a root with norm $v \in \R_{> 0}$ if and only if the coefficient $p$ of $f$ is located on a certain \textit{hypotrochoid curve} depending on $s,t,q$ and $v$, which is located in the
$\C$-slice $(T_A)_q$ of the parameter space $T_A$ (Theorem \ref{Thm:Hypotrochoid}). Secondly, we show that $f$ has two roots with identical norm if and only if $p$ is located on a union of rays in $\C \cong \R^2$, which yields the desired local description of the sets $U_j^A$ and their complements  (Theorems \ref{Thm:ComplementIn1Fan} and \ref{Thm:LocalStructureUalpha}). This union of rays in the $\C$-slice 
$(T_A)_q$ of the parameter space $T_A$ is precisely the geometric picture that corresponds to Egerv\'{a}ry's Theorem \ref{Thm:Egervary1} (2), which we already sketched at the end of 
Section~\ref{Sec:TrinomialsClassical}.

By combining both results, we show that $f$ has two roots of the same norm $v$ if $p$ is located on a singularity of the particular hypotrochoid corresponding to $v$ (Theorem \ref{Thm:MultipleNormsCorrespondToSingularities}). As a corollary we re-prove a classical result by Sommerville on the location of singularities on hypotrochoid curves (Corollary \ref{Prop:NodesonFan}). Furthermore, we show a result similar to Theorem \ref{Thm:Hypotrochoid} involving epitrochoids instead of hypotrochoids for the $\C^*$-slice $(T_A)_p$ of $T_A$ given by fixing the coefficient $p$ instead of $q$ (Theorem \ref{Thm:Epitrochoid}). Finally, we deduce some results about the discriminant of trinomials (Corollaries \ref{Cor:Discriminant} and \ref{Cor:LocalStructureUzero}).

Recall that a \textit{hypotrochoid} with parameters $R,r \in \Q_{> 0}$, $d \in \R_{> 0}$ satisfying $R \geq r$ is the parametric curve $\gamma$ in $\R^2 \cong \C$ given by
\begin{eqnarray}
	\gamma: [0,2\pi) \ra \C, \quad \phi \mapsto (R - r) \cdot e^{i \cdot \phi} + d \cdot e^{i \cdot \lf(\frac{r - R}{r}\ri) \cdot \phi}.
	\label{Equ:Hypotrochoid}
\end{eqnarray}
See Figure \ref{Fig:HypotrochoidExamples} for some examples and
references~\cite{Brieskorn:Knoerrer,Fladt} for detailed information.
Geometrically, a hypotrochoid is the trajectory of some fixed point with distance $d$ from the center of a circle with radius $r$ rolling in the interior of a circle with radius $R > r$ (see Figure \ref{Fig:HypotrochoidExplanation}).
Further note that hypotrochoids belong to the family of roulette curves, see \cite[Chapter 17]{Lockwood} for an overview. Hypotrochoids have certain well-known special instances themselves, in particular \textit{ellipses} (if $R = 2r$), \textit{hypocycloids} (if $d = r$) and \textit{rhodonea curves} (or \textit{rose curves}; if $R -r = d$). 
We say that a curve $\gamma$ is a \textit{hypotrochoid up to a rotation} if there exists some reparametrization $\rho_k : [0,2\pi) \ra [0,2\pi), \phi \mapsto k + \phi \mod 2 \pi$ with $k \in [0,2\pi)$, such that $\gamma \circ \rho_k^{-1}$ is a hypotrochoid.

\begin{figure}[ht]
\ifpictures
\includegraphics[width=0.3\linewidth]{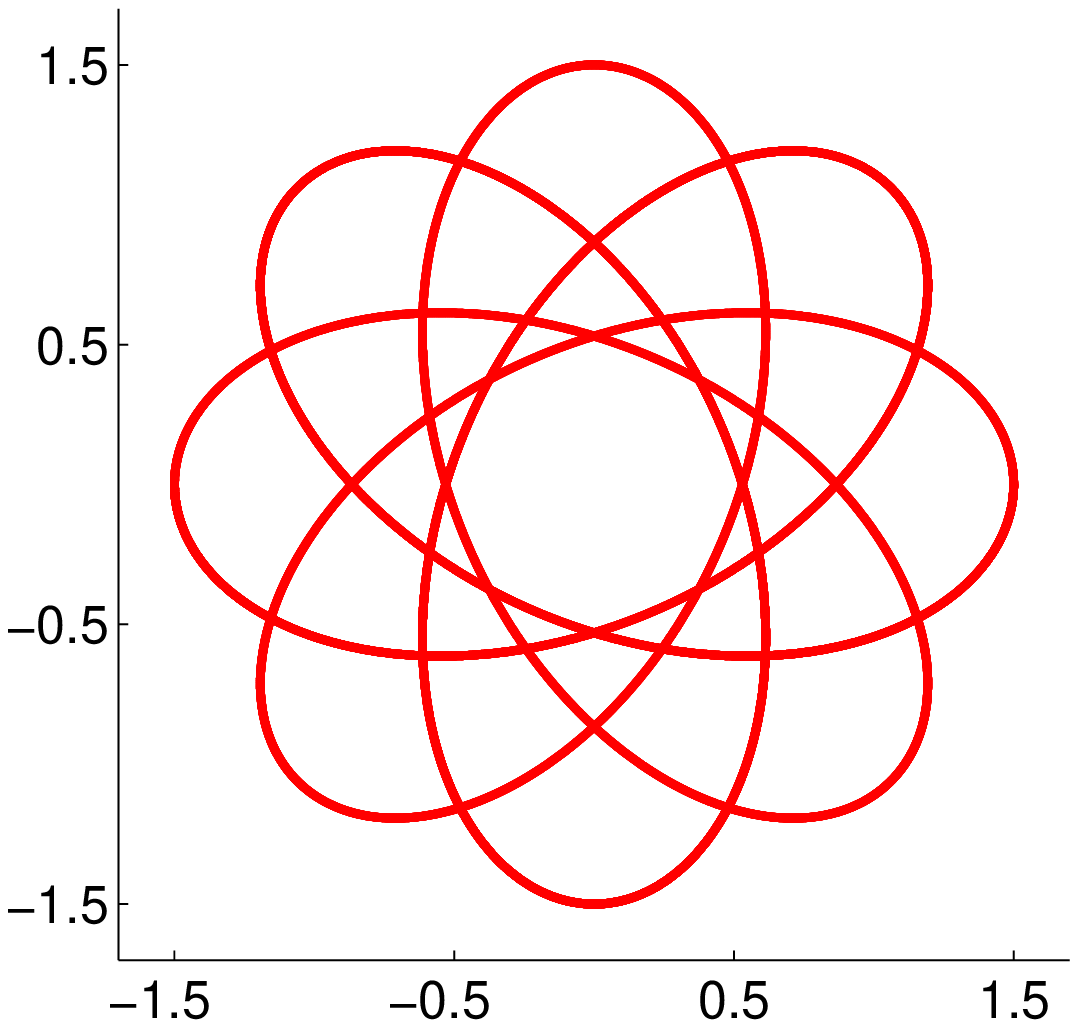}
\includegraphics[width=0.3\linewidth]{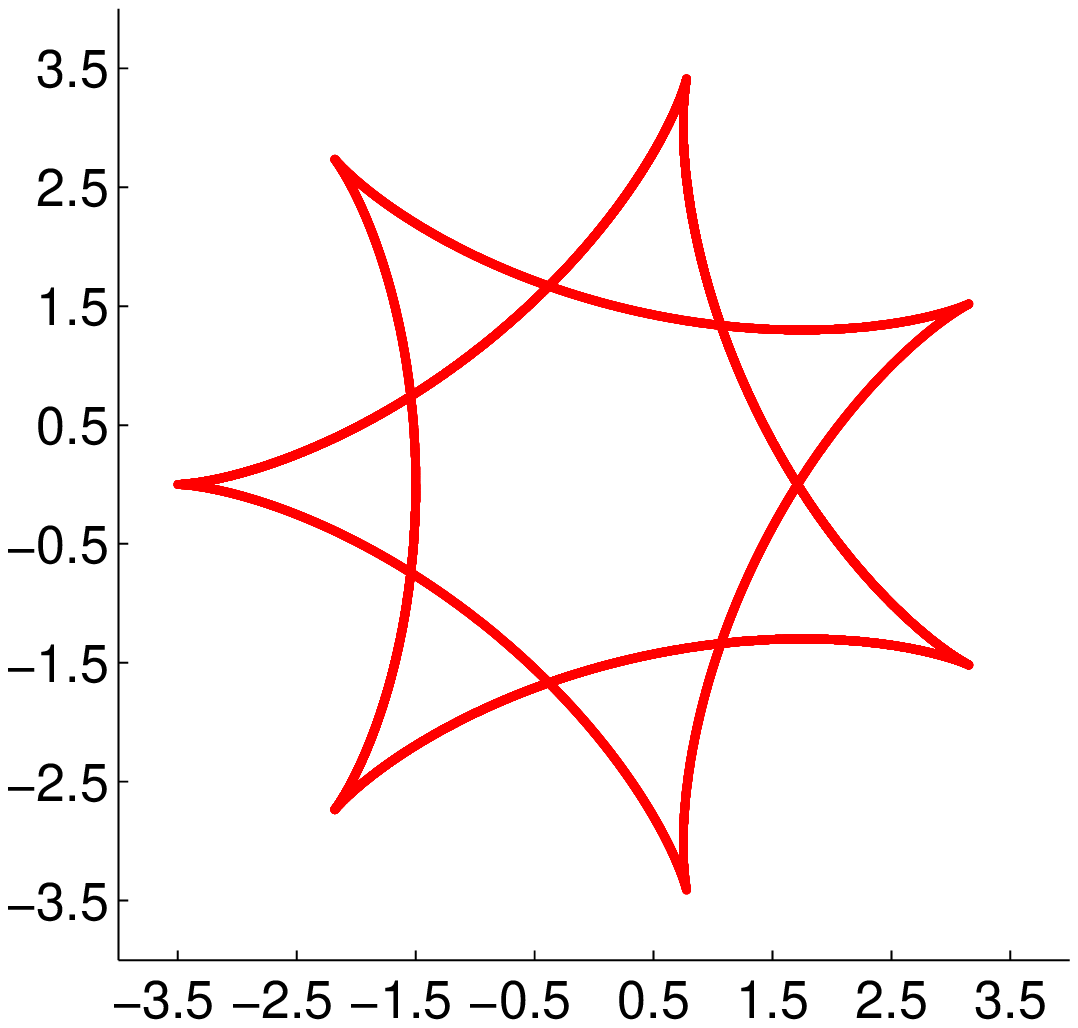}
\includegraphics[width=0.3\linewidth]{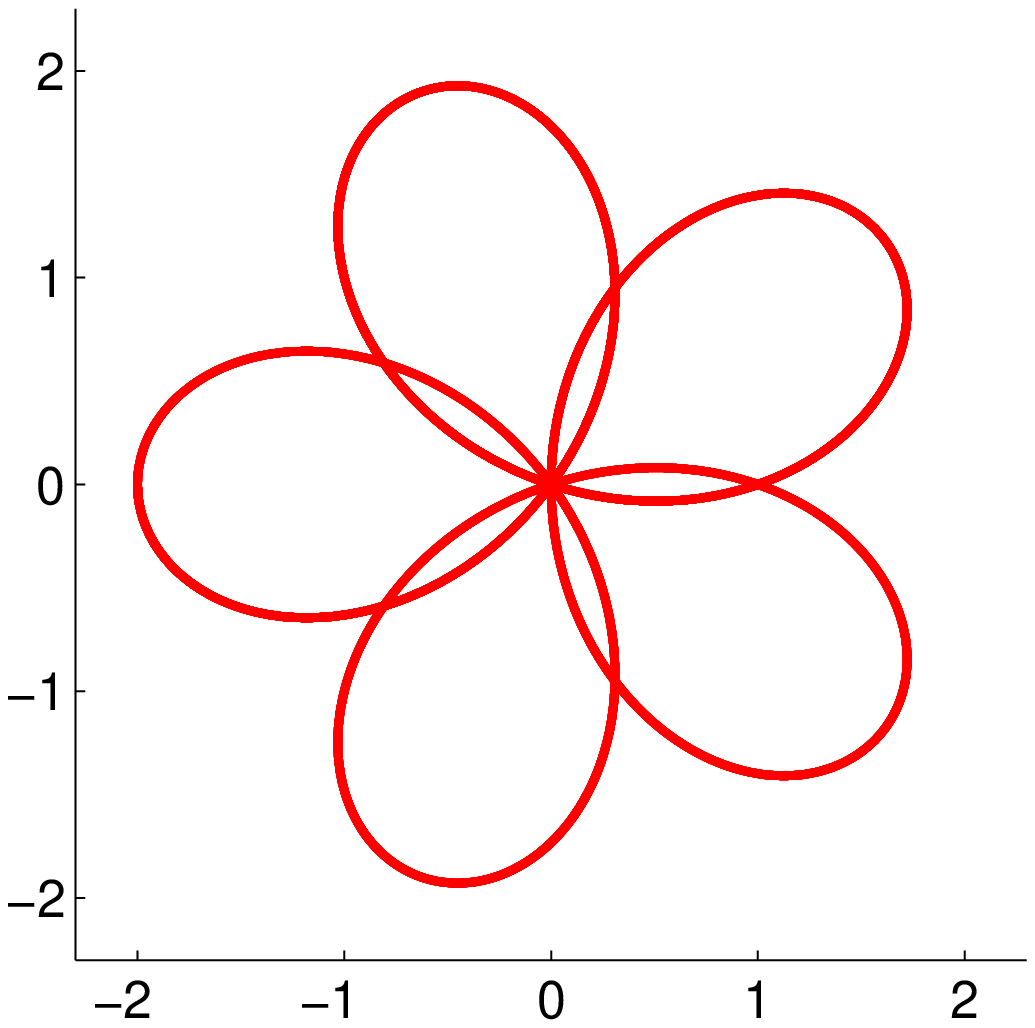}
\fi
\caption{Hypotrochoids for $(R,r,d) = (8/3,5/3,1/2),(7/2,5/2,5/2)$ and $(5,4,1)$. The second curve is a hypocycloid and the third one is a rhodonea curve, which are both special instances of
hypotrochoids.}
\label{Fig:HypotrochoidExamples}
\end{figure}

\begin{figure}[ht]
\ifpictures
\includegraphics[width=0.4\linewidth]{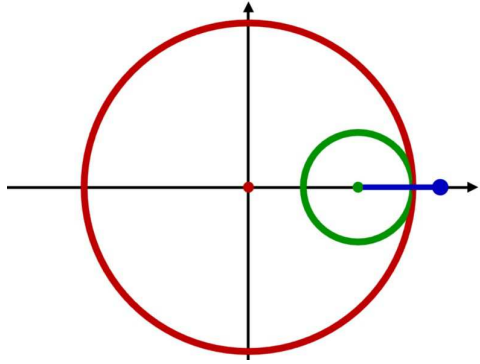}
\fi
\caption{A geometric explanation of a hypotrochoid. The green (small) circle with radius $r$ rolls inside the red (big) circle of radius $R$. The hypotrochoid describes the trajectory of the blue (fat) point with distance $d$ to the center of the green circle. The trajectory has finite length if $R /r \in \Q$.}
\label{Fig:HypotrochoidExplanation}
\end{figure}

We can now give the following answer to Problem~(A) from the Introduction. 

\begin{thm}
Let $f = z^{s+t} + p z^{t} + q$ with $p \in \C$ and $q \in \C^*$ be a trinomial and $v \in \R_{> 0}$. $f$ has a root of norm $v$ if and only if $p$ is located on a hypotrochoid up to a rotation with parameters $R = v^s / t \cdot  (t + s )$, $r = v^s / t \cdot s$ and $d = |q| \cdot v^{-t}$.
\label{Thm:Hypotrochoid}
\end{thm}

\begin{proof}
By~\eqref{Equ:FiberFunction2},
the trinomial $f$ has a root with norm $v \in \R_{> 0}$ if and only if 
the fiber function $f^{v}$ has non-empty zero set, i.e.,
\begin{eqnarray}
	p  + v^s \cdot e^{i \cdot s \cdot \phi} + |q| \cdot v^{-t} \cdot e^{i \cdot \lf(\arg(q) - t \cdot \phi\ri)}	 & =	& 0 \quad \text{ for some }
\phi \in [0,2 \pi).\label{Equ:FiberVarietyTrinomials}
\end{eqnarray}
Using $R - r = (t + s) \cdot v^s /t - s \cdot v^s / t = v^s$ as well as $(r - R) / r = -v^s / (s \cdot v^s / t) = - t/s$, and setting $\phi' = \phi \cdot s$, we obtain
\begin{equation}
\label{eq:hypo1}
 -p  \ \in \ \big\{(R - r) \cdot e^{i \cdot \phi'} + d \cdot e^{i \cdot \lf(\arg(q) + \frac{(r - R)}{r} \cdot \phi' \ri)} : \phi' \in
[0,2\pi)\big\}.
\end{equation}
By \eqref{Equ:Hypotrochoid}, the right hand side, and thus also its negative, is a hypotrochoid up to a rotation.
\end{proof}

\begin{exa}
Let $f = z^8 + p z^{3} + \frac{1}{2}$, $g = z^7 + p z^{2} + \frac{5}{2}$, $h = z^5 + p z + 1$. Then $f,g$ respectively $h$ has a root of norm one if and only if $p \in \C$ is located on the trajectory of the hypotrochoids with parameters $(R,r,d) = (8/3,5/3,1/2),(7/2,5/2,5/2)$ and $(5,4,1)$, 
depicted in Figure \ref{Fig:HypotrochoidExamples}.
\label{Exa:Hypotrochoids1}
\end{exa}

In order to tackle the question on trinomials with multiple roots of the same norm, i.e., Problem (B) from
the Introduction, we start from Bohl's Theorems to show the following initial fact.

\begin{prop}
Let $f = z^{s+t} + p z^{t} + q$ with $p,q \in \C^*$ and $v \in \R_{>0}$. Then 
at most two roots of $f$ have norm $v$.
\label{Prop:Atmost2Roots}
\end{prop}

\begin{proof}
Let $v \in \R_{> 0}$ such that there exists a root with norm $v$. Then there is a triangle $\Delta$ (possibly degenerated to a line segment) with edges of lengths $v^{s+t}$, $|p| v^{t}$ and $|q|$.
Let $k$ be the cardinality of $|\cV(f)| \cap \{z \in \C^* \, : \, |z| < v\}$.
By Theorem~\ref{Thm:Bohl2}, $k$ equals the number of integers in the open interval $I$ 
bounded by \eqref{Equ:Bohl1} and \eqref{Equ:Bohl2}.

Assume first that $\Delta$ is non-degenerate. Clearly, by Theorem \ref{Thm:Bohl2} an infinitesimal increase of $v$ can only increase the number of integers in the resulting open interval by at most~two. Hence, there can exist at most two roots of norm $v$.

If $\Delta$ degenerates to a line segment then one of the terms $v^{s+t}$, $|p| v^{t}$ and $|q|$
is the sum of the other two. First assume $v^{s+t} = |p| v^t + |q|$. Then the open interval $I$ with endpoints~\eqref{Equ:Bohl1} and \eqref{Equ:Bohl2} has length $s+t$,
and thus $I$ contains $s+t-1$ or $s+t$ integers. Hence, Bohl's Theorem~\ref{Thm:Bohl2}
asserts that there are $s+t-1$ or $s+t$ roots of norm less than $v$. Since $v^{s+t} > |p| v^t$,
infinitesimally increasing $v$ to $v+\eps$
leads to applicability of Bohls' first Theorem~\ref{Thm:Bohl1}, which then states that
there are exactly $s+t$ roots of norm less than $v+\eps$. Hence, in the case $v^{s+t} = |p| v^t + |q|$
there can only exist a single root of norm $v$. 

In the case $|q| = v^{s+t} + |p| v^t$, in Bohl's Theorem 3.2
we have $\alpha = \beta = 0$ and thus the open interval $I$
has length 0. Infinitesimally increasing $v$ to $v + \varepsilon$
leads to a non-degenerate triangle and to at most one integer
in the resulting open interval. There is at most one root with
norm less than $v+\epsilon$, and thus, by our initial assumption
on $v$, exactly one root with norm $v$.
In the case $|p| v^t = v^{s+t} + |q|$ infinitesimally increasing $v$ leads to a non-degenerate
triangle, and we can argue as in case of a non-degenerate triangle $\Delta$.
\end{proof}

For parameters $s,t \in \N^*$ and $q \in \C^*$, given by a trinomial $f = z^{s+t} + p z^{t} + q$, we define a union of rays
\begin{eqnarray}
		F(s,t,q)	& =	& \bigcup_{0 \le k \le 2(s+t)-1} \R_{\ge 0} \cdot e^{i \cdot (s \arg(q) +  k \cdot \pi )/(s+t)}.
\label{Equ:RaysofFan}
\end{eqnarray}

Note that $F(s,t,q) \subseteq \C$, where $\C$ can be regarded as the $\C$-slice $(T_A)_q$ of the augmented parameter space
\begin{eqnarray*}
	\tilde{T}_A & = & T_A \cup \{z^{s+t} + q \, : \, q \in \C^*\} \ \cong \ \C \times \C^* \, .
\end{eqnarray*}
We write $F(s,t,q) = F^{\odd}(s,t,q) \cup F^{\even}(s,t,q)$, where $F^{\odd}(s,t,q)$ and $F^{\even}(s,t,q)$ consists of the rays with odd $k$ (even $k$), see Figure \ref{Fig:HypotrochoidExample2}. If the context is clear, we just write $F^{\odd}$ and $F^{\even}$.

First we show that $F(s,t,q)$ is closely related to the question about multiple roots of the same norm.

\begin{thm}
Let $f = z^{s+t} + p z^{t} + q$ with $p \in \C$ and $q \in \C^*$ such that for two roots $a_1,a_2 \in \cV(f)$ we have $|a_1| = |a_2|$. Then $((s+t)\arg(p) - s\arg(q))/\pi \in \Z$ and thus $p \in F(s,t,q)$. In particular, $(U_j^A)^c \subseteq F(s,t,q)$ for every $1 \leq j \leq s+t-1$.
\label{Thm:ComplementIn1Fan}
\end{thm}

Note that Theorem \ref{Thm:ComplementIn1Fan} also covers the case $a_1 = a_2$.

\begin{proof}
Let $a_1,a_2 \in \cV(f)$ with $|a_1| = |a_2| = v \in \R_{> 0}$. We can apply Bohl's Theorem \ref{Thm:Bohl2} since there are roots with norm $v$ and hence the triangle $\Delta$ is well-defined. Theorem \ref{Thm:Bohl2} yields that for $v$ \textit{both} the numbers \eqref{Equ:Bohl1} and \eqref{Equ:Bohl2} are integers. Since these numbers are symmetric around the number
\begin{eqnarray*}
	k & = & \frac{(s+t)(\pi + \arg(p) - \arg(q)) - t(\pi -\arg(q))}{2\pi},
\end{eqnarray*}
we have $2k \in \Z$ and therefore $((s+t)\arg(p) - s\arg(q))/\pi \in \Z$. Now $p \in F(s,t,q)$ follows from
the Definition \eqref{Equ:RaysofFan} of the union of rays and $(U_j^A)^c \subseteq F(s,t,q)$ follows from the definition of $U_j^A$.
\end{proof}

If a trinomial has two roots which share the same norm $v$, then this fact has a nice interpretation in terms of the hypotrochoid curves given by our first Theorem \ref{Thm:Hypotrochoid}, since it corresponds to their singularities.

\begin{thm}
Let $f = z^{s+t} + p z^{t} + q$ with $p \in \C$ and $q \in \C^*$. There exist $ a_j,a_{j+1} \in \cV(f)$ with $|a_j| = |a_{j+1}| = v \in \R_{> 0}$ if and only if $p$ is a singular point of the hypotrochoid $f^{v} - p$ determined in~\eqref{eq:hypo1}.
In detail
\begin{enumerate}
 \item $f$ has two distinct roots with identical norm $v$ if and only if $p$ is located on a real double point of the hypotrochoid,
 \item $f$ has a root of multiplicity two with norm $v$ if and only if the corresponding hypotrochoid is a hypocycloid and $p$ is a cusp of it, and
 \item $f$ has more than two roots with norm $v$ if and only if $p = 0$ if and only if the 
  hypotrochoid is a rhodonea curve with a point of multiplicity $s+t$ in the origin.
\end{enumerate}
\label{Thm:MultipleNormsCorrespondToSingularities}
\end{thm}

\begin{proof}
There exist two roots $a_j,a_{j+1} \in \cV(f)$ with $|a_j| = |a_{j+1}| = v \in \R$ and $a_j \neq a_{j+1}$ if and only if there exist $\phi,\psi \in [0,2\pi)$ with $\phi \neq \psi$ and $f^{v}(\phi) = f^{v}(\psi) = 0$ i.e., equivalently, $f^{v}(\phi) - p = f^{v}(\psi) - p = -p$. This is the case if and only if the hypotrochoid $f^{v} - p$ attains the value $-p$ twice, i.e., it has a real double point at $-p \in \C$.

If $a_j = a_{j+1}$, i.e., $f$ has a double root, then we can consider this case as the limit of a family of trinomials given by the limit of $\phi \to \psi$ and therefore $a_j \to a_{j+1}$ in the upper case. That is,
the node at $f^{v}(\phi) - p$ degenerates to a cusp. 
Conversely, if the hypotrochoid $f^{v}(\phi) - p$ has a cusp, then it is a hypocycloid, i.e., $d = r$ in \eqref{Equ:Hypotrochoid} (see \cite{Brieskorn:Knoerrer}). Moreover, the cusp properties
$f^{v}(\phi) = 0$ and $\frac{\partial}{\partial \phi} f^v(\phi) = 0$ imply for the parameter values
$\phi^*$ of the cusps: $p = - \frac{s+t}{t} v^s e^{is\phi^{*}}$. Hence, 
the derivative $\frac{\partial f}{\partial z} = (s+t) z^{s+t-1} + tp z^{t-1}$ vanishes at 
$v \cdot e^{i \phi^{*}}$, and thus $f$ has a double root with norm $v$.

Assume finally $f$ has more than two roots with norm $v$. By Theorem \ref{Prop:Atmost2Roots} this is equivalent to $p = 0$ and hence $v = \sqrt[s+t]{|q|}$. Theorem \ref{Thm:Hypotrochoid} then implies $R - r = d$ for the corresponding hypotrochoid $f^{v} - p = f^{v}$. But $R - r = d$ means that the hypotrochoid is a rhodonea curve with a point of multiplicity $s+t$ in the origin. On the other hand, if a hypotrochoid has a point of multiplicity greater than two, then it is always a rhodonea curve and the singularity is located in the origin. Thus, $p = 0$ and $R - r = d$, and further, by Theorem \ref{Thm:Hypotrochoid}, $v^{s+t} = |q|$.
\end{proof}

As an immediate corollary of Theorems~\ref{Thm:ComplementIn1Fan} and \ref{Thm:MultipleNormsCorrespondToSingularities}, we regain a statement about hypotrochoids by Sommerville from 1920 \cite{Sommerville}.

\begin{cor}(Sommerville)
Let $\gamma: [0,2\pi) \ra \C$, $\phi \mapsto v^s e^{i \cdot s \phi} + |q| v^{-t} e^{i \cdot (\arg(q) - t \phi)}$ be a hypotrochoid. Then all singularities of $\gamma$ are located on $F(s,t,q)$.
\label{Prop:NodesonFan}
\end{cor}

\begin{rem}In order to see that Sommerville's result (stated in different notation) indeed matches
with the preceding corollary, express his variables $p$ and $q$ by $R$ and $r$ in his equation for $\theta$ in \cite[\S{10}, p.~390]{Sommerville} to obtain 
$\theta = \frac{k \pi r}{R} \ = \ \frac{k \pi s}{s+t}$ with $k \in \N$.
\end{rem}

To describe which subsets of $F(s,t,q)$ belong to the complement of a set $U_j^A$, we make use of the following observation about real trinomials.

%\timo{Folgendes Theorem inkl. Kommentar und Beweis um Bezug auf Resultat von Bertrand, Bihan und %Sottile erweitert. Bitte schau auch hier nochmal genau dr\"uber.}

\begin{thm}
Let $f = z^{s+t} + p z^{t} + q$ with $p,q \in \R^*$ and $\cV(f) = \{a_1\ldots,a_{s+t}\}$, such that $|a_1| \leq \cdots \leq |a_{s+t}|$. Assume $a_j$ is real. Then $j \in \{1,t,t+1,s+t\}$. Furthermore, if $a_t$ or $a_{t+1}$ is real, then $f$ is lopsided at every point in the interval $E_t(f) = \{w \in \R \ : \ \log|a_t| < w < \log|a_{t+1}|\}$. 
%Moreover, $f$ has at most three real roots, and for certain choices of $s,t,p$ and $q$ the trinomial $f$ %has exactly three real roots.
\label{Thm:LocationRealRoots}
\end{thm}

Consistent with D\'escartes' Rule of Signs,
Theorem \ref{Thm:LocationRealRoots} in particular implies that a real trinomial $f$ always has exactly one or three (respectively zero, two or four) real roots when $s+t$ is odd (respectively even). Indeed, since $s$ or $t$ is odd, a simultaneous application of D\'{e}scartes Rule on $f(z)$ 
and $f(-z)$ straightforwardly reveals that the case of four real roots (i.e., two positive ones and 
two negative ones) cannot occur.

%Furthermore, Theorem \ref{Thm:LocationRealRoots} refines the univariate case of a result by Bertrand, %Bihan and Sottile \cite{Bertrand:Bihan:Sottile} stating that a polynomial system with support a primitive %circuit has at most $2n + 1$ real solutions and that this bound is sharp. Recall in this context that a %finite subset $A \subset \Z^n$ is a \textit{circuit} if $A$ is a minimal affine dependent and that $A$ is %\textit{primitive} if it spans $\Z^n$; see \cite{Gelfand:Kapranov:Zelevinsky}.

\begin{proof}
Let $a_j$ be a real root of $f$. By Proposition \ref{Prop:Atmost2Roots} $a_j$ is a single or a double root. Furthermore, all the three terms $a_j^{s+t},p a_j^t$ and $q$ are real, and one of the monomials equals the sum of the two others. Hence, if we continuously increase the norm of the dominating monomial by $\eps > 0$, then the resulting polynomial $g$ is lopsided at $\log|a_j|$ (see Section \ref{SubSec:Amoebas}) and the ordering of the zeros is preserved (under the right labeling for the case that $f$ has a multiple real root). Hence, we can apply Bohl's Theorem \ref{Thm:Bohl1} (the analog for lopsidedness for univariate trinomials; see Section \ref{Sec:TrinomialsClassical}), which implies that $g$ contains $0,t$ or $s+t$ roots in the interior of the circle with radius $|a_j|$. Hence, the component $E_{j}(g)$ of the complement of the amoeba $\cA(g)$, which contains $\log|a_j|$, has order $0,t$ or $s+t$ (see Section \ref{SubSec:Amoebas}). 
Since $a_j$ is a root of $f$ it follows that $\log|a_j|$ is contained in the boundary of the closure of the components $E_{0}(f),E_{t}(f)$ or $E_{s+t}(f)$ of the complement (where $E_{t}(f)$ degenerates to the empty set in the case that $a_j$ is a multiple real root).

If $a_j$ is a single root then the definition of $E_t(f)$ directly implies
$j \in \{1,t,t+1,s+t\}$.

In case of a double root $a_j = a_{j+1} \neq 0$, the condition $f'(a_j) = 0$ 
implies $|a_j|^s = |p| \cdot t/(s+t)$ and 
(by considering the derivative of $z^s + p + q z^{-t}$)
$|a_j|^{s+t} = |q| \cdot t / s$. Division yields $|a|^{t} = |q| / |p|  \cdot (s+t) / s$,
and we conclude
\begin{equation}
  \label{eq:pajt}
  |p a_j^t| \ = \ |q| \frac{s+t}{s} \ = \ |q| + |a_j|^{s+t} \, .
\end{equation}
Hence, the complement component $E_j(g)$ of $\cA(g)$, which
contains $\log |a_j|$, has order $t$ with Theorem \ref{Thm:Bohl1}. With regard to the trinomial $f$, this shows $j=t$.
\end{proof}

With Theorem~\ref{Thm:LocationRealRoots} at hand, we now have all the tools to distinguish which subsets of $F(s,t,q)$ are part of the
complement of which $U_j^A$. Thus, together with Theorem \ref{Thm:ComplementIn1Fan}, the following theorem solves Problem (B) from the Introduction.

\begin{thm}
For fixed $q \in \C^*$, let 
$f_p = z^{s+t} + p z^t + q$ be a parametric family of trinomials with parameter $p \in \C$. 
For $j \in \{1,\ldots,s+t-1\} \setminus \{t\}$ the following holds.
\begin{equation}
\begin{array}{lrcl}
 \label{eq:parametric}
 \text{\emph{For} } s+j \text{\emph{ even we have:}} & f \in U_j^A & \text{if and only if} & p \notin F^{\even}. \\
 \text{\emph{For} } s+j \text{\emph{ odd we have:}} & f \in U_j^A & \text{if and only if} & p \notin F^{\odd}.
\end{array}
\end{equation}
In particular, the set $\{p \in \C^* \, : \, f_p \in U_{j}^A\}$ is not connected, and this
remains true for the set $\{p \in \C \, : \, f_p \in U_{j}^A\}$ when the sets $U_j^A$ are 
considered in $\C \times \C^*$.
For $U_t^A$, the conditions~\eqref{eq:parametric} hold as well, with the modification that 
we have additionally $f_p \in U_t^A$ if there exists a $v \in \R_{> 0}$ such that $f_p$ is lopsided with dominating term $p v^t$.
\label{Thm:LocalStructureUalpha}
\end{thm}

\begin{proof}
Let $a_1,\ldots,a_{s+t} \in \C^*$ denote the roots of $f_p$ (depending on $p$) with $|a_1| \leq \cdots \leq |a_{s+t}|$. By Theorem \ref{Thm:ComplementIn1Fan}, it suffices to consider the case $p \in F(s,t,q)$. 

By rescaling the norms of the roots, we can assume $|q| = 1$, and moreover, by uniformly
adding an offset to the arguments, we can even assume $q=1$.
For $p = 0$ every root has norm 1 and thus $f_0 \notin U_j^A$ for every $j \in \{1,\ldots,s+t-1\}$, i.e., we can always assume $p \in \C^*$. 
Since Bohl's Theorem~\ref{Thm:Bohl1} is only relevant for the case $j=t$, we first
consider the complementary cases $j \in \{1,\ldots,s+t-1\} \setminus \{t\}$.
Following the argument in the proof of Theorem \ref{Thm:ComplementIn1Fan}, the midpoint
\begin{eqnarray}
  \label{eq:midpointk}
	k & = & \frac{(s+t)(\pi + \arg(p) - \arg(q)) - t(\pi -\arg(q))}{2\pi},
\end{eqnarray}
of the interval in Bohl's Theorem \ref{Thm:Bohl2} is in $\frac{1}{2}\Z$ for $p \in F(s,t,q)$. Since the interval is symmetric around $k$ and the number of integers in the interval determines the number of roots of a particular norm, we have for all $j \in \{1,\ldots,s+t-1\} \setminus \{t\}$
\begin{eqnarray*}
 f \in U_j^A \text{ with } j \text{ even} & \text{iff} & k \notin \Z, \text{ and } \\
 f \in U_j^A \text{ with } j \text{ odd} & \text{iff} & k \in \Z.
\end{eqnarray*}
Thus, it only remains to show for which choices of $s,t$ and $\arg(p)$ we have $k \in \Z$. Since $\arg(q) = 0$, we have $p \in F(s,t,q)$ if and only if $\arg(p) = l \pi / (s+t)$ with $l \in \{0,\ldots,2(s+t)-1\}$. Hence, \eqref{eq:midpointk} simplifies to
$k = \frac{s + l}{2}$.
Since finally $p \in F(s,t,q)$ satisfies $p \in F^{\even}$ (respectively $p \in F^{\odd}$) if and only if $(s+t) \arg(p) / \pi$ is even (respectively odd), i.e., $l$ is even (respectively odd), the statement follows.

The non-connectedness of $U_{j}^A$ along the $\C$-slice $(T_A)_q$ for $q = 1$ follows directly from the fact that $\C \setminus F(s,t,q)^{\odd}$ respectively $\C \setminus F(s,t,q)^{\even}$ is not connected.

It only remains to investigate the special case $j = t$. The argument above remains valid for $j = t$ with the exception that, by the Theorems \ref{Thm:Bohl1} and \ref{Thm:LocationRealRoots}, we have additionally $f \in U_t^A$ if there exists a $z \in \C^*$ such that $f(z)$ is lopsided with dominating term $p z^t$.
\end{proof}

\begin{exa}We illustrate the different situations of the theorem.
\begin{enumerate}
	\item Let $f = x^5 + 6 x^2 + 1$, i.e., $s$ is odd, $t$ is even and $p \in F^{\even}(s,t,q)$.
        By Theorem~\ref{Thm:LocalStructureUalpha}, $f \in U_2^A \cap U_4^A$. 
        Since always $f \in U_0^A$ and $f \in U_{s+t}^A$, this gives
        $f \in U_0^A \cap U_2^A \cap U_4^A \cap U_5^A$.
        We verify this by determining the absolute values of $\cV(f)$ approximately:
	 $$0.4082, 0.4082, 1.8030, 1.8030, 1.8462.$$
	\item Let $f = x^5 - 6 x^2 + 1$, i.e., $s$ is odd, $t$ is even and $p \in F^{\odd}(s,t,q)$. 
         By Theorem~\ref{Thm:LocalStructureUalpha}, $f \in U_1^A \cap U_3^A$, 
         and clearly at the point $v=1$ the function $f$ is lopsided with dominating term $6 v^2$.
         Hence, altogether, $f \in U_0^A \cap U_1^A \cap U_2^A \cap U_3^A \cap U_5^A$.
     The approximate absolute values of $\cV(f)$ do verify this:
	$$0.4060, 0.4106, 1.7849, 1.8332, 1.8332.$$
\item $f = x^5 + 6 x^3 + 1$, i.e., $s$ is even, $t$ is odd and $p \in F^{\even}(s,t,q)$. 
      By Theorem~\ref{Thm:LocalStructureUalpha}, $f \in 
      U_1^A \cap U_3^A$.
      Hence, altogether, $f \in U_0^A \cap U_1^A \cap U_3^A \cap U_5^A$.
      The absolute values of $\cV(f)$ are approximately
	$$0.5416, 0.5546, 0.5546, 2.4498, 2.4498.$$
	i.e., $a_{1}$ is the unique real root and thus $f \in U_0^A \cap U_1^A \cap U_3^A \cap U_5^A$.
	\item Let $f = x^4 + 0.5 x^1 + 1$, i.e., $s$ is odd, $t$ is odd and $p \in F^{\even}(s,t,q)$. 
              By Theorem~\ref{Thm:LocalStructureUalpha}, $f \in U_2^A$, so that
      altogether $f \in U_0^A \cap U_2^A \cap U_4^A$.
      The absolute values of $\cV(f)$ are approximately
	$$0.916, 0.916, 1.091, 1.091.$$
\end{enumerate}
\end{exa}

In the following we investigate the discriminant $D$ of trinomials $f$, Recall that the discriminant 
is a polynomial function depending on the coefficients of $f$, i.e.,
$D:T_A \to \C$, which vanishes when $f$ has a double root (see, e.g.,
\cite{Gelfand:Kapranov:Zelevinsky}). For general monic polynomials 
of degree $n$, the discriminant $D(f)$ can
be defined by $D(f) = (-1)^{n(n-1)/2} \resultant(f,f')$, where $\resultant(f,f')$ 
is the resultant of $f$ and its derivative $f'$.
An explicit formula for the discriminant of univariate trinomials is well-known; see also part (3) of Egerv\'{a}ry's Theorem \ref{Thm:Egervary1}:

\begin{lemma}(Greenfield, Drucker \cite{Greenfield:Drucker})
Let $f = z^{s+t} + p z^t + q$ be a trinomial with $p,q \in \C$ and $\gcd(s,t) = 1$. Then the discriminant $D(f)$ of $f$ is given by
\begin{eqnarray*}
	D(f) & = & (-1)^{\frac{(s+t)(s+t-1)}{2}} q^{s} \left(q^s (s+t)^{s+t} - (-1)^{s+t} p^{s+t} s^s t^t \right).
\end{eqnarray*}
\label{Lem:Discriminant}
\end{lemma}

From our earlier statements, we can conclude additional information about $D$.

\begin{cor}Any trinomial lying on the hypersurface defined by the discriminant $D$ lies on the boundary
of the complement $(U_t^A)^c$ of $U_t^A$. 
In particular, if a trinomial $f =  z^{s+t} + p z^t + q$ with $p,q \in \C^*$
has a double root $a_j = a_{j+1}$, then $j = t$.
\label{Cor:Discriminant}
\end{cor}

\begin{proof} 
In the proof of Part (2) of Theorem \ref{Thm:MultipleNormsCorrespondToSingularities} we have seen that for fixed $q \in \C^*$ there is a unique choice $|p^*|$ for the norm of $p$ such that $a_j = a_{j+1}$.
For every $p'$ with
$|p'| > |p^*|$, the resulting trinomial $f_{p'}$ is lopsided in the interval $E_t(f_{p'})$.
Thus, $f_{p'} \in U_t^A$ for every $|p'| > |p^*|$.

It remains to show that if $f$ has a double root $a_j = a_{j+1} \neq 0$, then $|p| = |p^*|$
and $j = t$. 
Since the argument for the case of a double real root in the proof of Theorem \ref{Thm:LocationRealRoots}
also holds in the complex case, we can deduce $j=t$ and 
$|p a_j^t| \ = \ |q| \frac{s+t}{s} \ = \ |q| + |a_j|^{s+t}$,
which shows that $|p|$ coincides with the unique choice $|p^*|$ introduced above.
Altogether, $\mathcal{V}(D) \subseteq \partial((U_t^A)^c)$.
\end{proof}

\begin{cor}For fixed $q \in \C^*$, let $f_p = z^{s+t} + p z^t + q$ be a parametric
family of trinomials with parameter $p \in \C$. Then:
\[
\begin{array}{lrcl}
 \label{eq:parametric2}
 \text{\emph{For} } s+j \text{\emph{ even we have:}} & f \in U_t^A & \text{if and only if} & p \notin F^{\even} \cap B_r(0), \\
 \text{\emph{For} } s+j \text{\emph{ odd we have:}} & f \in U_t^A & \text{if and only if} & p \notin F^{\odd} \cap B_r(0),
\end{array}
\]
where $r = |q|^{s/(s+t)} \left((t/s)^{s/(s+t)} + (s/t)^{t/(s+t)}\right)$ and $B_r(0)$
is the closed disk with radius $r$ around the origin. 
Furthermore, the $s+t$ intersection points of $F(s,t,q)$ with the
boundary of this disk equal the intersection of $(T_A)_q$
with the $\C$-slice of the discriminant zero set $\cV(D)$ obtained by fixing $q$.
\label{Cor:LocalStructureUzero}
\end{cor}

\begin{rem}
For the special case of a quadratic equation, $s=t=1$, the corollary yields that there exist
two roots of the same norm if and only if $p \in \C^*$ satisfies
$\arg(p)  = \frac{1}{2}(\arg(q) + \pi k)$
for some $k \in \{0,2\}$ 
and $|p| \le 2 |q|^{1/2}$.
\end{rem}

\begin{proof}
By Theorem \ref{Thm:LocalStructureUalpha} $(U_{t}^A)^c$ restricted to the $\C$-slice given by fixing $q \in \C^*$ is the subset of $F(s,t,q)^{\odd}$ (respectively $F(s,t,q)^{\even}$) where $f$ is not lopsided with dominating term $pz^t$. 
It is well-known that lopsidedness is independent of arguments of coefficients (see \cite[Proposition 5.2]{Theobald:deWolff:Genus1} or \cite[Proposition 4.14]{deWolff:Diss}), it holds on an open subset of $U_t^A$ (by definition of lopsidedness; see Section \ref{SubSec:Amoebas}) and is kept under increasing of $|p|$. Thus,  $(U_{t}^A)^c$ along the $\C$-slice $(T_A)_q$ for fixed $q$ is given by $F(s,t,q)^{\odd}$ (respectively $F(s,t,q)^{\even}$) intersected with a closed disk. By the reasoning in the proof of Corollary \ref{Cor:Discriminant}, the boundary of
this disk is given by the choice of $|p| > 0$ such that $f = z^{s+t} + |p| z^t + |q|$ has a 
double root $a_t = a_{t+1}$ (and thus the discriminant vanishes).
Since $f_{|p|+\eps}$ is lopsided at $a_t$ for every $\eps > 0$, 
we know from that proof
$|a_t| = \sqrt[s+t]{|q| t/s}$ and, from~\eqref{eq:pajt}, 
\[
  |p| |q t / s|^{t/(s+t)} \ = \ |q| t/s + |q| \, ,
\]
whence $|p| = |q|^{s/(s+t)} \left((t/s)^{s/(s+t)} + (s/t)^{t/(s+t)}\right)$.
\end{proof}

\begin{exa}
As in Example \ref{Exa:Hypotrochoids1}, let $f = z^8 + p z^{3} + 0.5$, $g = z^7 + p z^{2} + 2.5$ and $h = z^5 + p z + 1$. Then $f,g,h$ have two roots with the same norm if and only if $p$ is located on the blue (dotted) union of rays in Figure \ref{Fig:HypotrochoidExample2}.
\label{Exa:Hypotrochoids2}
\end{exa}

\begin{figure}[ht]
\ifpictures
\includegraphics[width=0.32\linewidth]{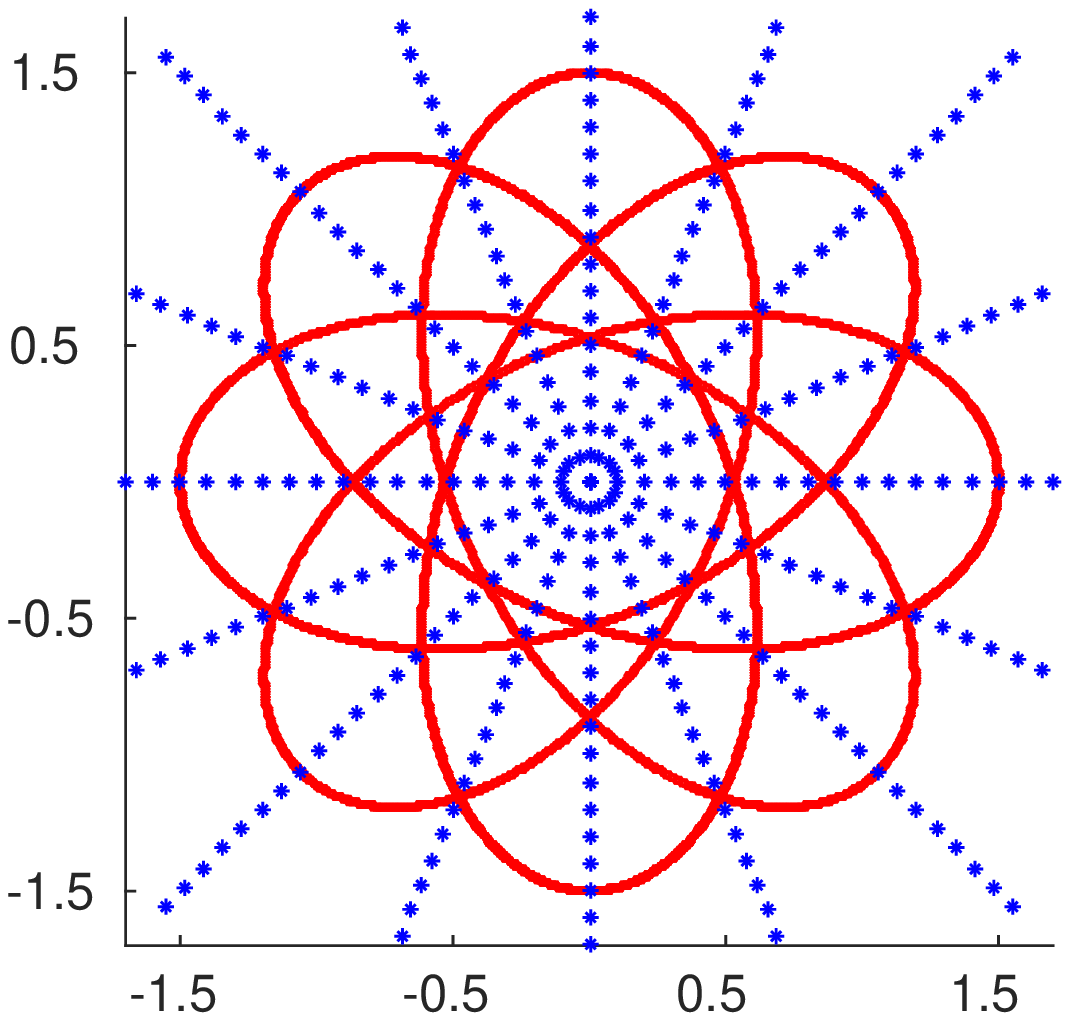}
\includegraphics[width=0.32\linewidth]{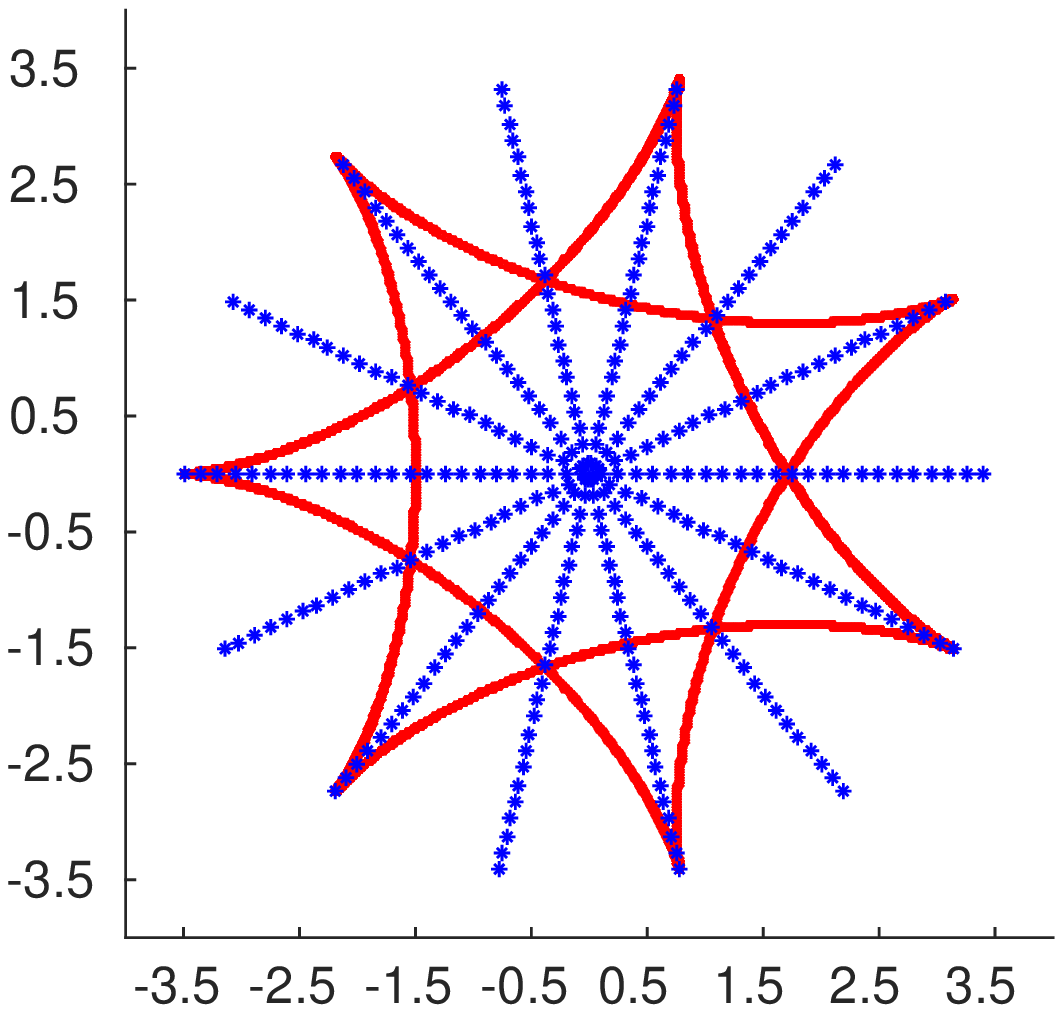}
\includegraphics[width=0.32\linewidth]{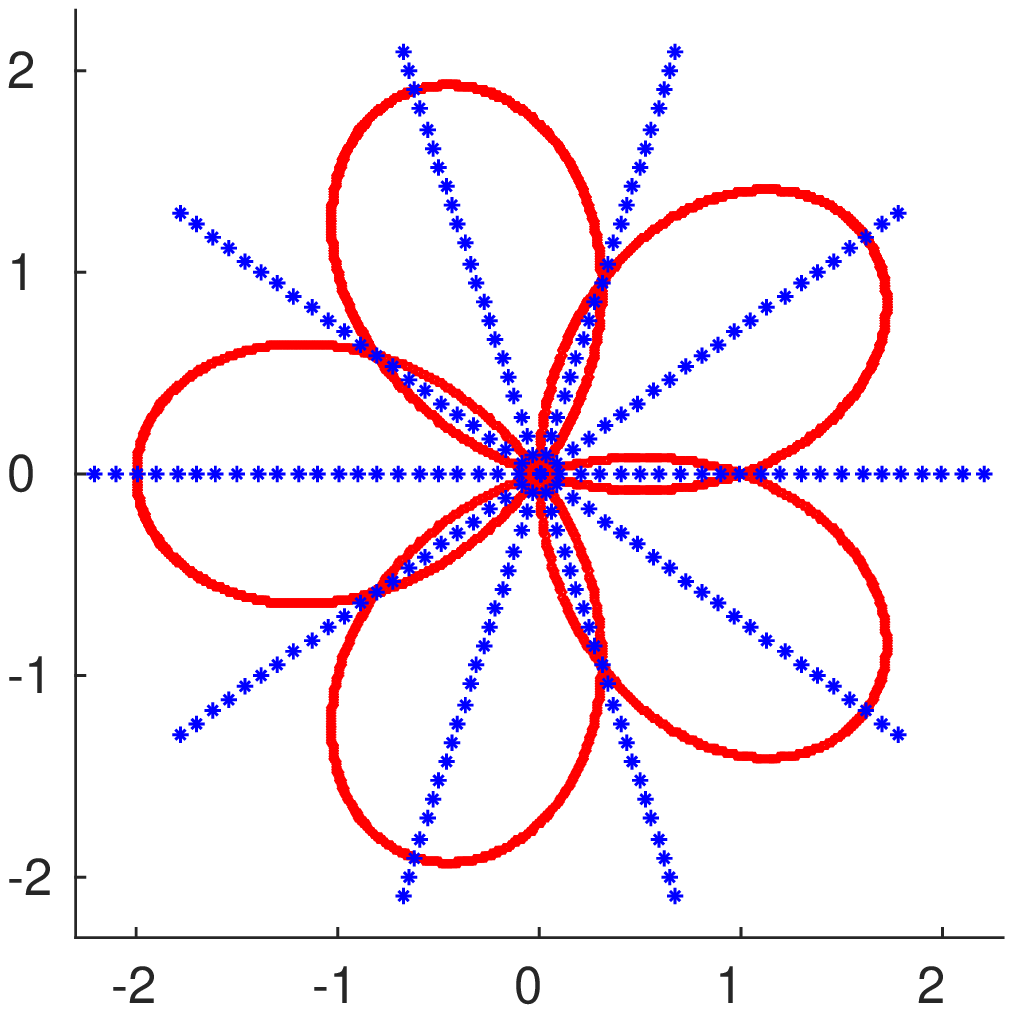}
\fi
	\caption{Three hypotrochoids:    The trajectory of the fiber functions $f^{1} - p$, $g^{1} - p$ and $h^{1} - p$ for the norm $v = 1$ and trinomials $f = z^8 + p z^{3} + 0.5$, $g = z^7 + p z^{2} + 2.5$ and $h = z^5 + p z + 1$ with their corresponding union of rays $F(s,t,q)$ (the blue dotted rays). The hypotrochoids and $F(s,t,q)$ are located in a complex plane, which is a $\C$-slice of the parameter space $\tilde{T}_A$.}
 	\label{Fig:HypotrochoidExample2}
\end{figure}

Finally, we investigate the local situation along a $\C^*$-slice of $T_A$ as in Theorem \ref{Thm:Hypotrochoid} but for fixed $p$ and variable $q$.
 That is, we like to know when a trinomial $f$ has a root of norm $v \in \R_{> 0}$ in dependence of the choice of $q \in \C^*$. It turns out that the natural objects needed to describe this situation are \textit{epitrochoids}, which can be regarded as canonical counterparts of hypotrochoids.

An \textit{epitrochoid} with parameters $R, r \in \Q_{> 0}, d \in \R_{> 0}$ is a parametric curve $\gamma \in \R^2 \cong \C$ given by
\begin{eqnarray}
\gamma: [0,2\pi) \ra \C, \quad \phi \mapsto (R + r) \cdot e^{i \cdot \phi} - d \cdot e^{i \cdot \lf(\frac{R + r}{r}\ri) \cdot \phi}.
	\label{Equ:Epitrochoid}
\end{eqnarray}

We say that a curve $\gamma$ is an \textit{epitrochoid up to a rotation} if there exists some reparametrization $\rho_k : [0,2\pi) \ra [0,2\pi), \phi \mapsto k + \phi \mod 2 \pi$ with $k \in [0,2\pi)$, such that $\gamma \circ \rho_k^{-1}$ is an epitrochoid. 

In Figure \ref{Fig:EpitrochoidExamples} we give some examples of epitrochoids. Like hypotrochoids, epitrochoids also belong to the family of roulette curves and have certain well-known special instances themselves, in particular \textit{epicycloids} given by $d = r$ and \textit{limacons} given by $R = r$
(see, e.g., \cite{Brieskorn:Knoerrer}).
\begin{figure}[ht]
\ifpictures
\includegraphics[width=0.32\linewidth]{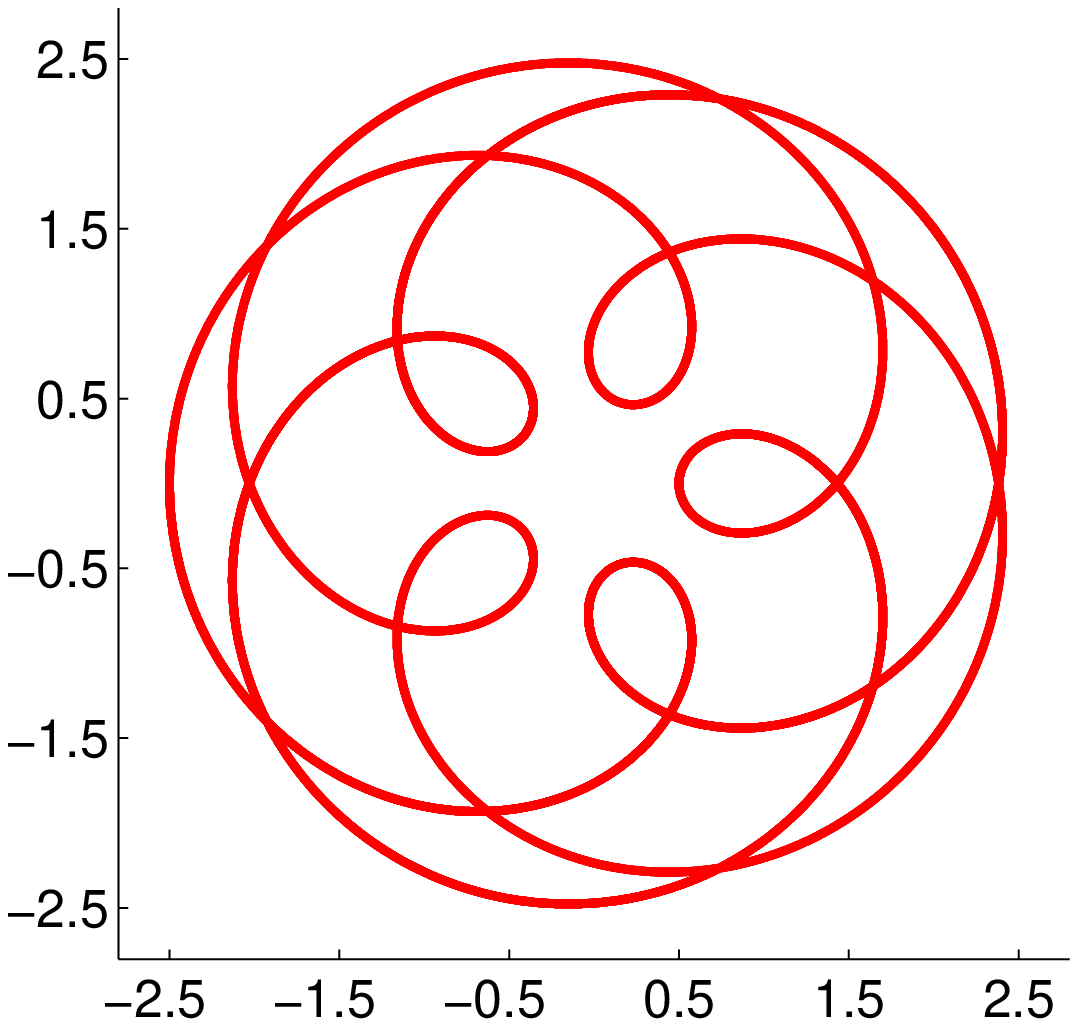}
\includegraphics[width=0.32\linewidth]{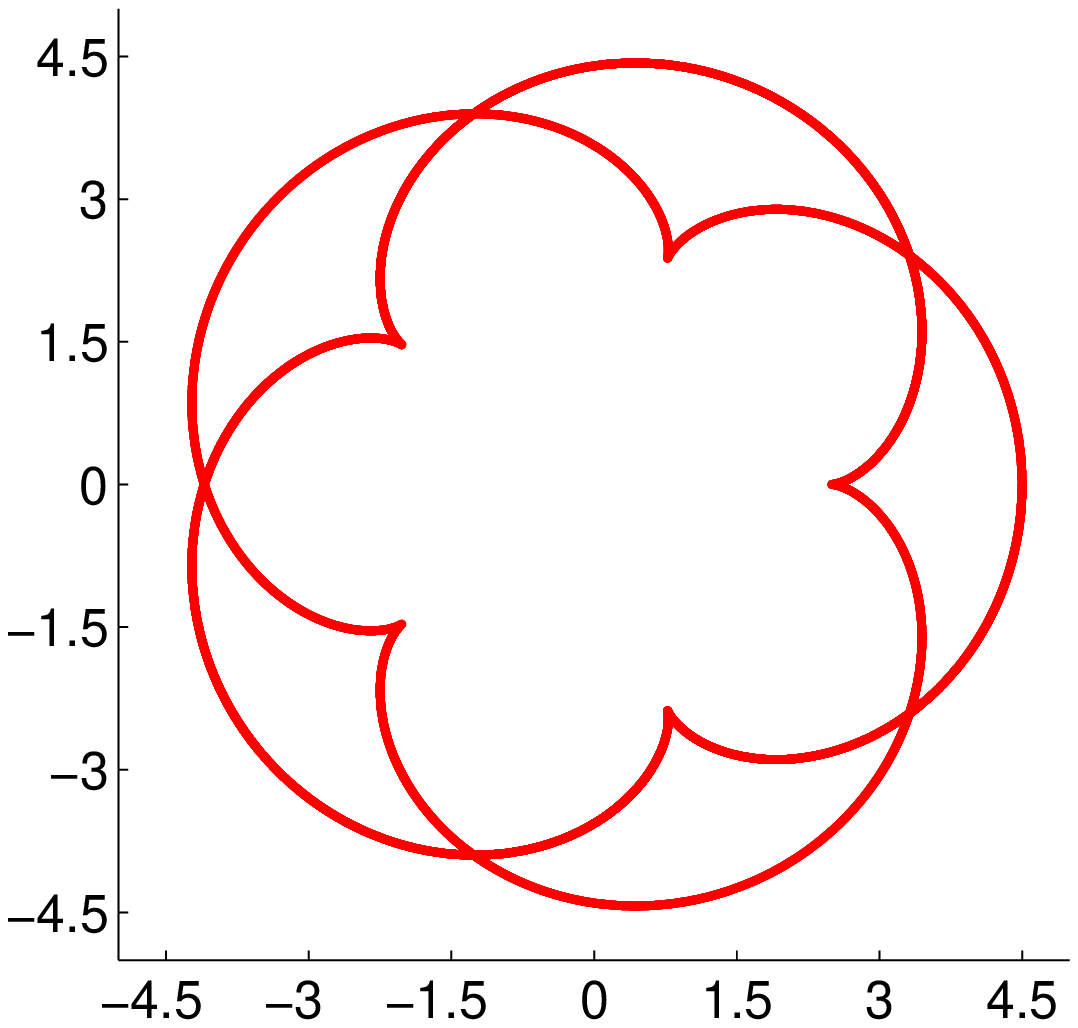}
\includegraphics[width=0.32\linewidth]{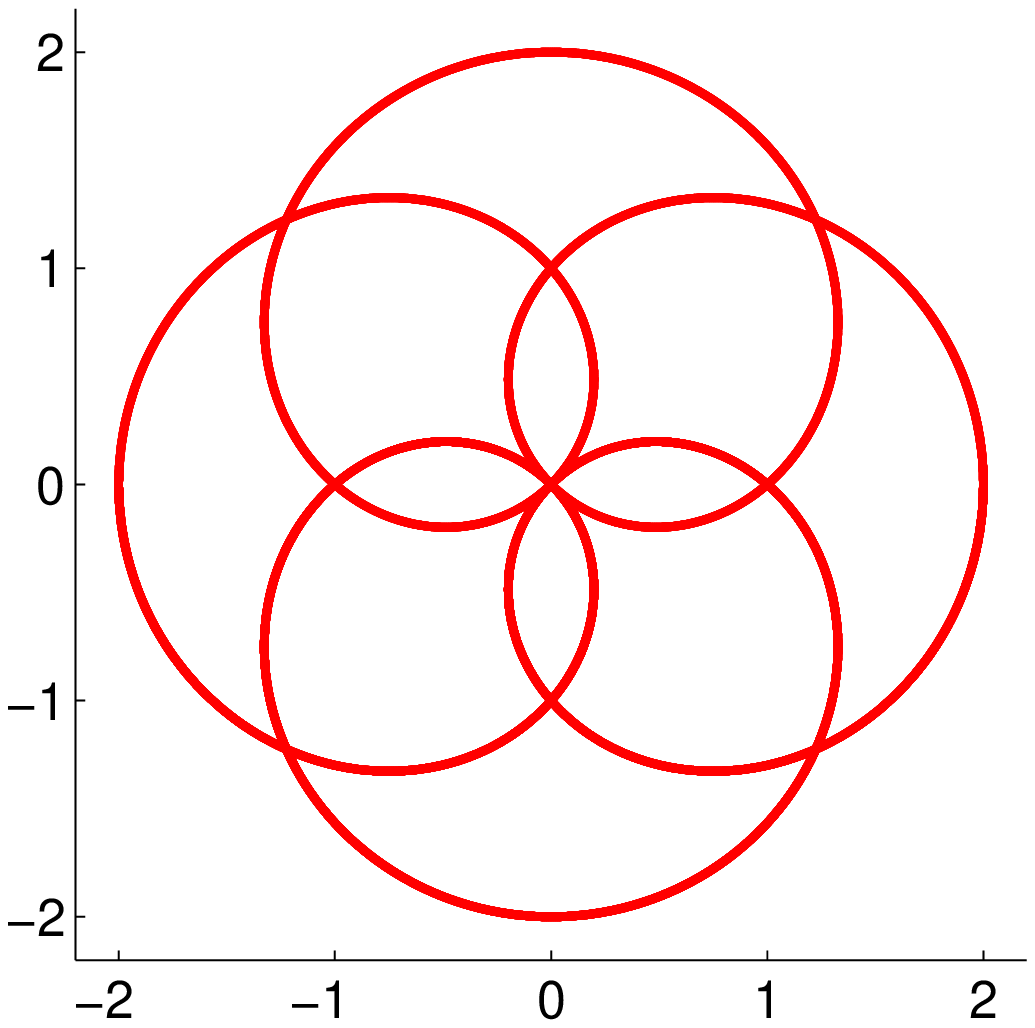}
\fi
\caption{Epitrochoids for $(R,r,d) = (15/16,9/16,1), (5/2,1,1)$ and $(4/5,1/5,1)$ in the complex plane. As we have $r = d$ in the second example, we see that the epitrochochoid in fact is an epicycloid.}
\label{Fig:EpitrochoidExamples}
\end{figure}

Geometrically, an epitrochoid is the trajectory of some fixed point with distance $d$ from the center of a circle with radius $r$ rolling along the exterior of a circle with radius $R$. 

With epitrochoids we can now obtain a counterpart to Theorem \ref{Thm:Hypotrochoid} regarding 
Problem (A) from the Introduction.

\begin{thm}
Let $f = z^{s+t} + p z^{t} + q$ with $p \in \C$ and $q \in \C^*$ be a trinomial and $v \in \R_{> 0}$. $f$ has a root of norm $v \in \R_{> 0}$ if and only if $q$ is located, up to a rotation, on an epitrochoid with parameters $R = v^t \cdot |p| \cdot  s/(s + t)$, $r = v^t \cdot |p| \cdot t/(s+t)$ and $d = v^{s+t}$.
% % % $R = v^t \cdot |p| \cdot  (s+t-1)/(s + t)$, $r = v^t \cdot |p| \cdot 1/(s+t)$ and $d = v^{s+t}$.
\label{Thm:Epitrochoid}
\end{thm}

\begin{proof}
By the same argument as in the proof of Theorem \ref{Thm:Hypotrochoid} we know that $f$ has a root of norm $v \in \R_{> 0}$ if and only if the corresponding fiber function $f^{v}$ satisfies
\begin{eqnarray}
	q & = & |p| v^{t} \cdot e^{i \cdot \left(\pi + \arg(p) + t \cdot \phi\right)} - v^{s+t} \cdot e^{i \cdot (s+t) \cdot \phi} \quad \text{ for some } \phi \in [0,2 \pi).\label{Equ:FiberVarietyTrinomials2}
\end{eqnarray}
The right hand side of the equation describes an epitrochochoid 
Namely, by definition of $R$ and $r$ we have
$R + r = v^t \cdot |p| \cdot (s/(s+t) + t/(s+t)) = v^t \cdot |p|$ and
\begin{eqnarray*}
	\frac{R+r}{r} & = & \frac{v^t \cdot |p|}{v^{t} \cdot |p| \cdot t/(s+t)} \ = \ \frac{s+t}{t}.
\end{eqnarray*}
Thus, by setting $\phi' = t \phi$,  \eqref{Equ:FiberVarietyTrinomials2} is equivalent to
\begin{eqnarray*}
	q & = & (R+r) \cdot e^{i \cdot \left(\pi + \arg(p) + \phi\right)} - d \cdot e^{i \cdot \frac{R+r}{r} \cdot \phi} \quad \text{ for some } \phi \in [0,2 \pi).
\end{eqnarray*}
By the definition of the epitrochoid in \eqref{Equ:Epitrochoid} the statement follows.
\end{proof}

\begin{exa}
As canonical counterpart to Example \ref{Exa:Hypotrochoids1} we investigate the trinomials $f = z^8 + \frac{3}{2} z^3 + q$, $g = z^7 - \frac{7}{2} z^2 + q$ and $h = z^5 + z + q$. Then $f,g$ and $h$ have a root of norm one if and only if $q$ is located on the epitrochoids with parameters $(R,r,d) = (15/16,9/16,1), (5/2,1,1)$ and $(4/5,1/5,1)$, which we depicted in Figure \ref{Fig:EpitrochoidExamples}.
\end{exa}

\section{The Topological Structure of the Parameter Space of Trinomials}
\label{Sec:TrinomialsTopology}

The aim of this section is to determine the fundamental groups of the sets 
$U_j^A\subseteq T_A$ and their complements $(U_j^A)^c \subseteq T_A$.
That is, we provide an answer to Problem (C).

As a main result we show that for every $j \in \{1,\ldots,s+t-1\} \setminus \{t\}$ both the set $U_j^A \subseteq T_A$ and its complement $(U_j^A)^c = T_A \setminus U_j^A$ as well as the discriminant zero set $\cV(D)$ can be deformation
retracted to a torus knot $K(s+t,s)$. $(U_t^A)^c$ can be deformation retracted to $\cV(D)$. Thus, all these sets are connected, but not simply connected and have fundamental group $\Z$ (see Theorem~\ref{Thm:TopologyTrinomials}).
 $U_t^A$ has a different topology; it has fundamental group $\Z^2$ (see Theorem~\ref{Thm:TopologyMiddleTerm}). 
Finally, we describe the amoeba and the coamoeba of $\cV(D)$ (Corollary \ref{Cor:DiscriminantAmoeba}). For background information about torus knots, see \cite{Burde:Zieschang, Hatcher}.

Note that, by Theorem~ \ref{Thm:LocalStructureUalpha}, for $j \neq t$ the sets $U_j^A $ are \textit{not} connected along a $\C$-slice given by a fixing $q \in \C^*$.

As a motivation and to provide an intuition about the structure of $T_A$ we give an example showing that a set $U_j^A \subseteq T_A$ can be connected although none of the sets $U_j^A$ intersected with a $\C$-slice of $T_A$ given by fixing $q \in
\C^*$ is connected.
\begin{exa}
Let $f = z^3 + 1.5 \cdot e^{i \cdot \arg(p)} z + e^{i \cdot \arg(q)}$ with $f = (z - a_1)(z - a_2)(z - a_3)$ and $|a_1| \leq |a_2| \leq |a_3|$. Assume, we want to construct a path $\gamma$ in $T_A$ from $(p_1,q_1) = (1.5 \cdot e^{i \cdot \pi/2},1)$ to $(p_2,q_2) = (1.5 \cdot e^{-i \cdot \pi/6},1)$ such that $\gamma \subseteq U_2^A$, i.e., $|a_2| \neq |a_3|$ for every point on $\gamma$. Theorem \ref{Thm:LocalStructureUalpha} implies that this is impossible if $\arg(q)$ remains constant for every point on $\gamma$. Similarly, we do not have $|a_2| \neq |a_3|$ for all points on an arbitrary path on a $\C$-slice of $T_A$ given by fixing $\arg(p)$ or in general by fixing an affine linear relation between $\arg(p)$ and $\arg(q)$. We illustrate this by investigating two closed paths starting and ending at $(p_1,q_1)$ given by
$\eta_1: [0,1] \rightarrow T_A, k \mapsto (1.5 \cdot e^{i \cdot \pi/2},e^{i \cdot 2k\pi})$ and $\eta_2: [0,1] \rightarrow T_A, k \mapsto (1.5 \cdot e^{i (1/4 + k) \cdot 2\pi},e^{i \cdot 2k\pi})$ 
(see Figure \ref{Fig:ZeroPathes}). But there exists a path $\gamma$ as desired given by $\gamma: [0,1] \rightarrow T_A, k \mapsto (1.5 \cdot e^{i (1/4 + 2k/3) \cdot 2\pi},e^{i \cdot 2k\pi})$ from $(p_1,q_1)$ to $(p_2,q_2)$ that is completely contained in $U_2^A$ (see Figure \ref{Fig:ZeroPathes}).
\end{exa}

\begin{figure}[ht]
\ifpictures
	\includegraphics[width=0.32\linewidth]{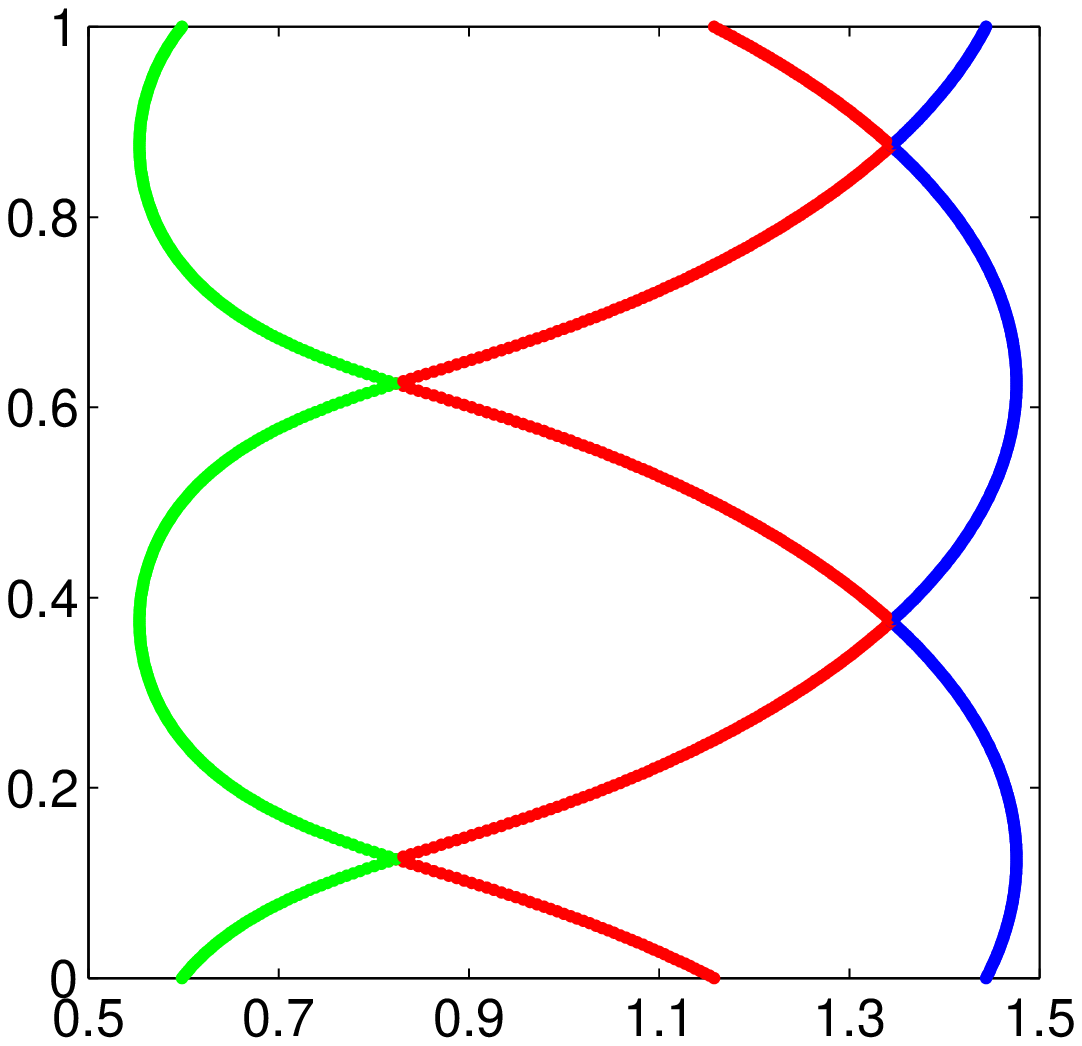}
	\includegraphics[width=0.32\linewidth]{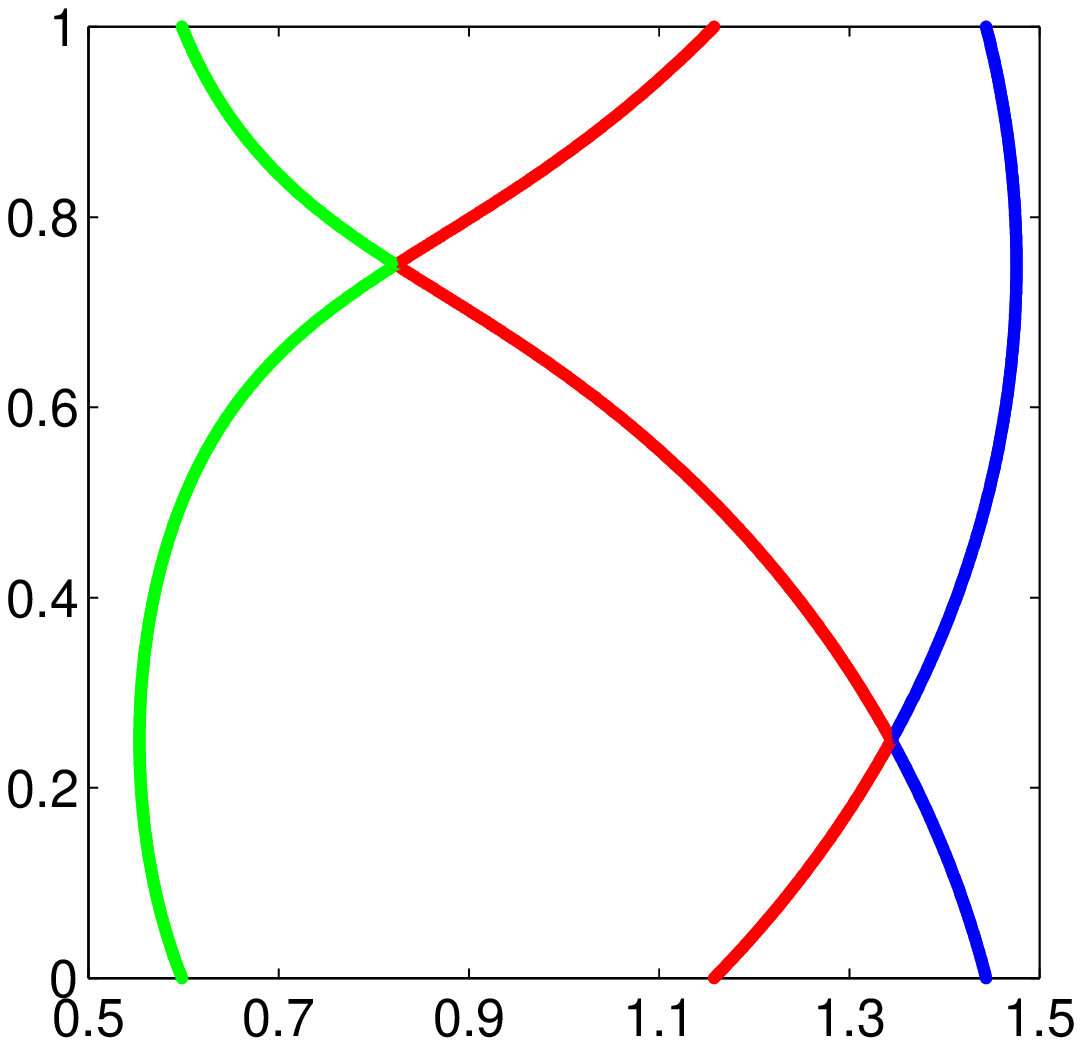}
	\includegraphics[width=0.32\linewidth]{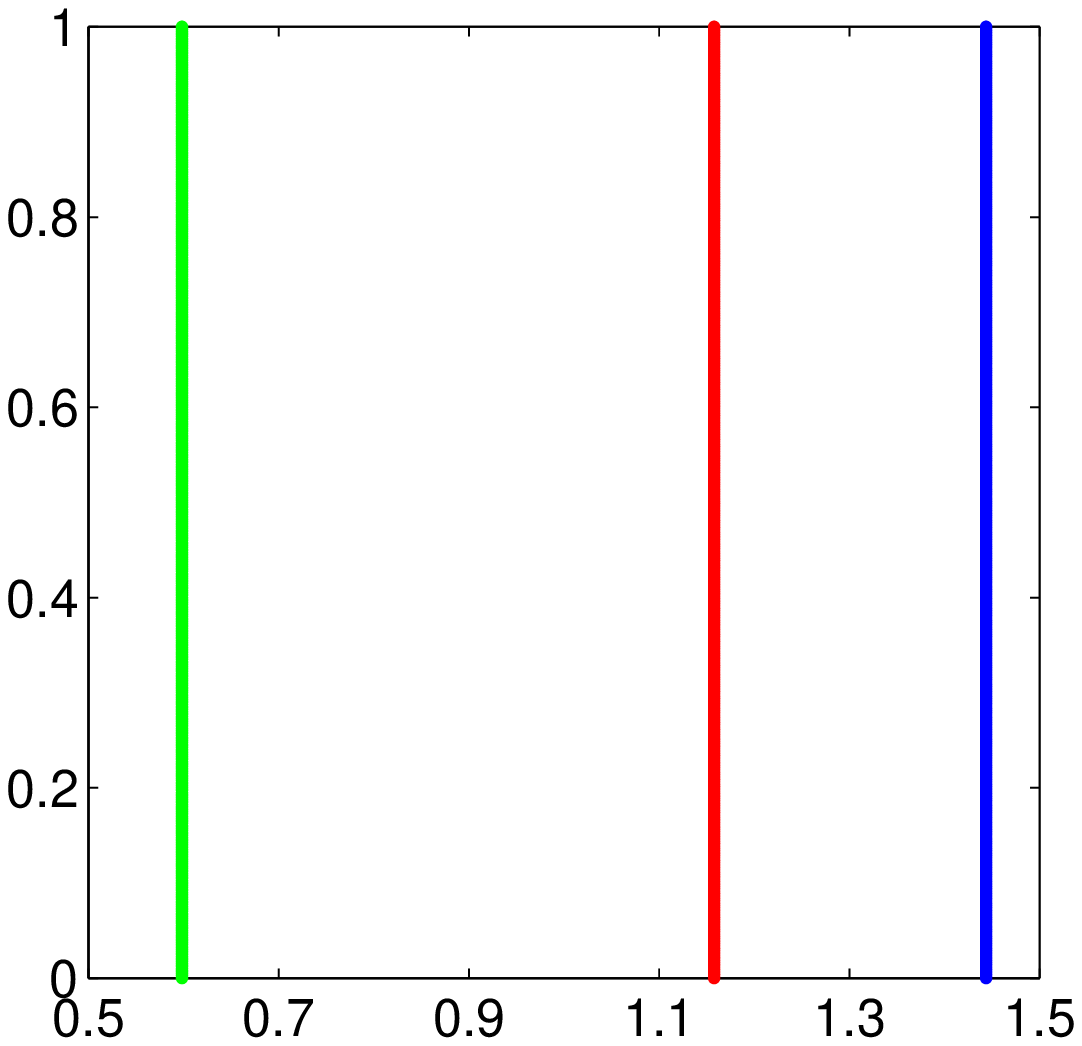}
\fi
	\caption{The norms of the roots of $f$ along the paths $\eta_1$, $\eta_2$ and $\gamma$ (from left to right). The vertical axis is $k$, the horizontal axis represents the norms of particular roots.}
	\label{Fig:ZeroPathes}
\end{figure}

First, we investigate the sets $U_j^A \subseteq T_A$ and their complements for $j \neq t$. Initially, we show that for these sets it suffices to investigate the situation of fixed $|p|$ and $|q|$.

\begin{lemma}
Every set $U_j^A \subseteq T_A$ with $j \in \{1,\ldots,s+t-1\} \setminus \{t\}$ and its complement $(U_j^A)^c$ can be deformation retracted to a subset $\wh U_j^A$ (respectively $(\wh U_j^A)^c$) of the standard torus $T_{(1,1)} = \{\lf(e^{i \cdot \arg(p)},e^{i \cdot \arg(q)}\ri) \ : \ p,q \in \C^*\} \subseteq T_A$. 
\label{Lemma:Deformation1Ualp}
\end{lemma}

The idea of the lemma is that for containment in $U_j^A$ (with $j \neq t$) the norms of the coefficients $p$ and $q$ are irrelevant. Thus, we can deform the complete space to the standard torus where $p$ and $q$ have norm one. 

\begin{proof} Recall from \eqref{Equ:FiberBundleArgMap} in Section \ref{SubSec:Fibers} that $T_A \cong (\C^*)^2$ comes with a fiber bundle $\R^2 \ra T_A \ra (S^1)^2$ given by the $\Arg$ map. 

Let $h$ be the homeomorphism given by applying $\Log_\C$ on the coefficients of polynomials in $T_A$. Let $T_{(1,1)} = h^{-1}((S^1)^2 \times \{0\}) = \{\lf(e^{i \cdot \arg(p)},e^{i \cdot \arg(q)}\ri) \ : \ p,q \in \C^*\}$ be the canonical embedding of the standard torus in $T_A$ . We investigate the homotopy 
\begin{eqnarray}
	& & F: T_A \times [0,1] \ra T_A, \quad ((p,q),l) \mapsto \lf(\frac{p}{(1-l)+l \cdot |p|},\frac{q}{(1-l)+l \cdot |q|}\ri). \label{Equ:Homotopy}
\end{eqnarray}
Obviously, $F$ is the identity for $l=0$ and $F$ is the projection from $T_A\times [0,1]$ to $T_{(1,1)}$ for $l = 1$. Recall that by Theorem \ref{Thm:LocalStructureUalpha} $U_j^A$ is invariant under changing $|q|$ and, for $q \in \C^*$ fixed, it holds that $(p,q) \in U_j^A$ implies $(\lam p,q) \in U_j^A$ for every $\lam \in \R_{>0}$. Since, by construction, every $(e^{i \cdot \phi_1},e^{i \cdot \phi_2}) \in T_{(1,1)}$ satisfies
\begin{eqnarray*}
	F^{-1}((e^{i \cdot \phi_1},e^{i \cdot \phi_2})) \ \cong \ \{(p,q) \in (\C^*)^2 \ : \ \arg(p) = \phi_1, \arg(q) = \phi_2\} \ = \ \Arg^{-1}((\phi_1,\phi_2)),
\end{eqnarray*}
$F$ respects the fiber bundle structure of $T_A$ described above and hence $F_{|U_j^A}$ is indeed a deformation retraction of $U_j^A$ to a subset of $T_{(1,1)}$. For the complements of the $U_j^A$ the argument works the same way.
\end{proof}

In the following we say that two polynomials $f,g \in \C[z]$ are \textit{equivalent} if the complements of their amoebas have the same components (with respect to the order map), i.e.,
\begin{eqnarray}
	f \sim g	& \Longleftrightarrow	& f \in U_j^A \text{ if and only if } g \in U_j^A \text{ for all } j \in \{0,\ldots,s+t\}. \label{Equ:Equivalence}
\end{eqnarray}

\begin{lemma} 
Let $f = z^{s+t} + p z^{t} + q$ with $p,q \in \C^*$. Then $f \sim g$ for every $g$ on the path  $\gamma_{(p,q)}: [0,1] \ra T_A, \quad \phi \mapsto (p \cdot  e^{i \cdot 2 \pi s \phi / (s+t)},
q \cdot e^{i \cdot 2 \pi \phi})$. In particular, we have $f \sim g$ for polynomials $f$ and $g$ with coefficient vectors $(p,q)$ and $(p \cdot e^{i \cdot 2\pi s/(s+t)},q)$ located on the torus
\begin{eqnarray}
	T_{(|p|,|q|)} & = & \{(|p| e^{i \cdot 2 \pi \phi}, |q| e^{i \cdot 2 \pi \psi}) \ : \ \phi,\psi \in [0,1)\} \subseteq T_A.
\end{eqnarray}
\label{Lem:GluingalongS1}
\end{lemma}

\begin{proof}
Let $f = \gamma_{(p,q)}(0)$ and $g = \gamma_{(p,q)}(\phi)$ for some $\phi \in (0,1]$. By Definition \eqref{Equ:RaysofFan}, $f$ is located on a particular ray of $F(s,t,q)$ if and only if $g$ is located on the corresponding ray of $F(s,t,q \cdot e^{i \cdot 2\pi \phi})$. 
Thus, by Theorem \ref{Thm:LocalStructureUalpha}, $f \in U_j^A$ if and only if $g \in U_j^A$ for all $j \in \{0,\ldots,s+t\}$. Note particularly that the equivalence also holds for $j = t$ since the coefficients of all trinomials on $T_{(|p|,|q|)}$ have the same norm and lopsidedness is either given for every point on a torus $T_{(|p|,|q|)}$ or for none (see also \cite[Proposition 5.2]{Theobald:deWolff:Genus1} or \cite[Proposition 4.14]{deWolff:Diss}). Hence, by \eqref{Equ:Equivalence}, we have $f \sim g$. Since $\phi$ was arbitrarily chosen, the statement follows.
\end{proof}

By considering the union of rays~\eqref{Equ:RaysofFan} for varying $\arg(q)$ we
make the transition from the sliced version of $U_j^A$ to the
global version. In the following, we provide an explicit parameterization
of the torus version $\wh{U}_j^A$ of $U_j^A$.

Let $\Z_m = \Z / m \Z$ for $m \in \N^*$. Note that since $\gcd(s,t) = 1$ we have $2\pi \cdot ks / (s+t) \equiv 0 \mod 2 \pi$ if and only if $k \in (s+t)\Z$. And since $(p,q) \sim (p \cdot e^{i \cdot 2 \pi s/(s+t)},q)$, every $U_j^A$ and $(U_j^A)^c$ with $j \in \{1,\ldots,s+t-1\} \setminus \{t\}$ is invariant under the group $\Z_{s+t}$ acting on $T_{(|p|,|q|)}$ by
\begin{eqnarray}
  \label{eq:groupaction}
	& & \ast: \Z_{s+t} \times T_{(|p|,|q|)} \ \ra \ T_{(|p|,|q|)}, \label{Equ:GroupAction} \\
	& & (k,(|p| \cdot e^{i \cdot 2 \pi \phi},|q| \cdot e^{i \cdot 2\pi \psi})) \ \mapsto \ (|p| \cdot e^{i \cdot 2 \pi (\phi + ks/(s+t))},|q| \cdot e^{i \cdot 2\pi \psi}). \nonumber
\end{eqnarray}
For $(p,q) \in T_{(|p|,|q|)}$ and $k \in \Z_{s+t}$ we denote the image of the group action as $k \ast (p,q)$. Note that this group action is conformal with the regular torus action of $T_A \cong (\C^*)^2$ on itself.

Let $\gamma_{(p,q)}$ be a path on $T_{(|p|,|q|)}$ as defined in Lemma \ref{Lem:GluingalongS1} and $k \ast (p,q)$ denote the image under the group action of $\Z_{s+t}$ on $T_{(|p|,|q|)}$ introduced in~\eqref{eq:groupaction}. Note that $\gamma_{(p,q)}$ has startpoint $(p,q)$ and endpoint $(p \cdot e^{i \cdot 2\pi \cdot s/(s+t)},q)$. Thus, for every $k \in \N$ the endpoint of $\gamma_{k \ast (p,q)}$ is the starting point of $\gamma_{(k+1) \ast (p,q)}$. 

Let $\rho_{(p,q)}$ denote the path on the torus $T_{(|p|,|q|)}$ given by
\begin{eqnarray}
	\rho_{(p,q)} & = & \gamma_{(s+t-1) \ast (p,q)} \circ \gamma_{(s+t-2) \ast (p,q)} \circ \cdots \circ \gamma_{1 \ast (p,q)} \circ \gamma_{0 \ast (p,q)} \label{Equ:PathRho}
\end{eqnarray}
for some $(p,q) \in T_{(|p|,|q|)} \subseteq T_A$. 

To denote paths $\rho_{(p,q)}$ with $|p| = |q| = 1$ on the standard torus $T_{(1,1)} = \{\lf(e^{i \cdot \arg(p)},e^{i \cdot \arg(q)}\ri) \ : \ p,q \in \C^*\} \subseteq T_A$ we also write $\rho_{(\arg(p),\arg(q))}$ with slight abuse of notation.

We observe that in case of $s+j$ even the curve
$\rho(0,0)$ parameterizes the set $(\wh{U}_j^A)^c$, and in case of
$s+j$ odd the curve $\rho(\pi/(s+t),0)$ parameterizes $(\wh{U}_j^A)^c$.

Note that $\rho_{(\arg(p),\arg(q))}$ is closed and not contractable on $T_{(1,1)}$ by construction. See Figure \ref{Fig:Torus} for a visualization. But it has an even stronger, well-known structure, as we show in the following corollary.

\begin{figure}[ht]
\ifpictures
	\includegraphics[width=0.6\linewidth]{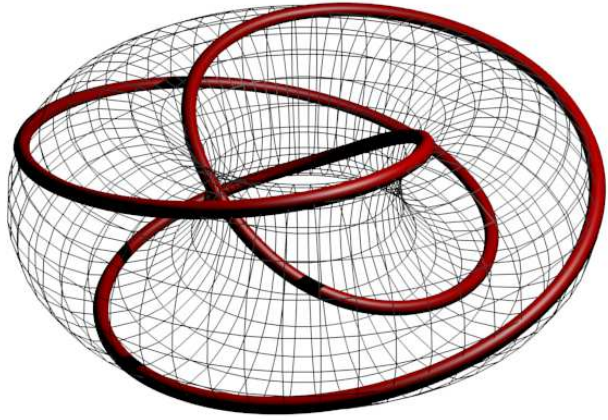}
\fi
	\caption{The curve $\rho(0,0)$ for $s = 2$ and  $s+t = 3$ on a torus. By Corollary \ref{Cor:TorusKnot} it corresponds to the $K(3,2)$ -- the trefoil knot.}
	\label{Fig:Torus}
\end{figure}

\begin{cor}
Every path $\rho_{(\arg(p),\arg(q))}$ is homeomorphic to the torus knot $K(s+t,s)$.
\label{Cor:TorusKnot}
\end{cor}

\begin{proof}
By construction, every $\rho_{(\arg(p),\arg(q))}$ is homeomorphic to $\rho_{(0,0)}$. $\rho_{(0,0)}$ is given by $S^1 \to (S^1)^2$, $\phi \mapsto (e^{i s \phi \cdot 2\pi}, e^{i (s+t) \phi \cdot 2\pi})$. This is a closed curve on $(S^1)^2$ such that the meridians are intersected $(s+t)$ respectively $s$ times.
Thus, it is the torus knot $K(s+t,s)$ (see \cite[p. 46 et seq.]{Burde:Zieschang}, see also \cite{Hatcher}).
\end{proof}

With the construction of $\rho_{(\arg(p),\arg(q))}$ we can describe the sets $\wh U_j^A$ and its complements on the standard torus $T_{(1,1)}$.

\begin{lemma}
Let $j \in \{1,\ldots,s+t-1\} \setminus \{t\}$. Then
\begin{equation*}
\begin{array}{ll}
 \text{\emph{For} } s+j \text{\emph{ even we have:}} & 
	\rho_{(0,0)} = (\wh U_j^A)^c \text{ and } \rho_{(\pi/(s+t),0)} \text{ is a deformation retract of } \wh U_j^A. \\
 \text{\emph{For} } s+j \text{\emph{ odd we have:}} & 
\rho_{(\pi/(s+t),0)} = 
(\wh U_j^A)^c \text{ and } \rho_{(0,0)} \text{ is a deformation retract of } \wh U_j^A.
\end{array}
\end{equation*}
 \label{Lemma:Deformation2Ualp}
\end{lemma}

\begin{proof} We consider the case $s + j$ be even and $j \neq t$. Recall that by Lemma \ref{Lemma:Deformation1Ualp}, $\wh U_j^A$ is the deformation retract of $U_j^A$ to a subset of the standard torus $T_{(1,1)}$. Hence, it suffices to show that a given point on the standard torus
$T_{(1,1)}$ belongs to $(U_j^A)^c$ if and only if it is is located on $\rho_{(0,0)}$. 
By Theorem \ref{Thm:LocalStructureUalpha}, $f = z^{s+t} + pz^t + q$ does not belong to $U_j^A$ if and only if $p \in F(s,t,q)^{\even}$. And it follows from \eqref{Equ:RaysofFan} that
if $f$ is additionally in $T_{(1,1)}$, i.e., $|p| = |q| = 1$, then $f \in U_j^A$ if and only if $\arg(p) \neq (\arg(q)s + 2\pi k)/(s+t)$ for $k \in \{1,\ldots,s+t\}$. By definition of $\rho_{(\arg(p),\arg(q))}$ these are exactly the points on $\rho_{(0,0)} \subseteq T_{(1,1)}$.

Now, we investigate $\wh U_j^A = T_{(1,1)} \setminus (\wh U_j^A)^c = T_{(1,1)} \setminus \rho{(0,0)}$. Since $\rho_{(\pi/(s+t),0)}$ is obtained from $\rho_{(0,0)}$ by the translation $(\arg(p),\arg(q)) \mapsto (\arg(p) + \pi/(s+t),\arg(q))$, we have $\rho_{(\pi/(s+t),0)} \subseteq \wh U_j^A$. We investigate the homotopy
\begin{eqnarray*}
\wh F: T_{(1,1)} \times [0,1] & \ra & T_{(1,1)},\\
((\arg(p),\arg(q)),l) & \mapsto & \lf(\arg(p) + l \cdot \lf(\frac{\arg(q)s + \pi}{s+t} - \lf(\arg(p) \mod \frac{2\pi}{s+t} \ri)\ri),\arg(q)\ri).	
\end{eqnarray*}

Obviously, we have $\wh F(\wh U_j^A,0) = \wh U_j^A$ and since $(\arg(p),\arg(q)) \in \rho_{(\pi/(s+t),0)} \Lera \arg(p) = (\arg(q)s + (1+2k)\pi)/(s+t)$ for $k \in \{1,\ldots,s+t\}$ we have $\wh F(\wh U_j^A,1) = \rho_{(\pi/(s+t),0)}$ (see Figure \ref{Fig:HomotopyTrinomials}).

Since $\wh F$ is continuous in $\arg(q)$ and the second coordinate of the image is independent of $l$, it suffices to prove the homotopy for the first image coordinate for an arbitrary, fixed $\arg(q)$. For a fixed $\arg(q)$ the set $\wh U_j^A$ is given by all $\arg(p) \neq (\arg(q)s + 2\pi k)/(s+t)$ for $k \in \{1,\ldots,s+t\}$. Thus, it consists of $s+t$ separated, open segments with midpoints $(\arg(q)s + (1+2k)\pi)/(s+t)$, where $k \in \{0,\ldots,s+t-1\}$. Each segment is contracted to its midpoint by $\wh F$ and hence $\wh F$ indeed is a deformation retraction of $\wh U_j^A$ to $\rho_{(\pi/(s+t),0)}$. For $s + j$ odd with $j \neq t$ the proof works analogously.
\end{proof}

\begin{figure}[ht]
\ifpictures
\includegraphics[width=0.25\linewidth]{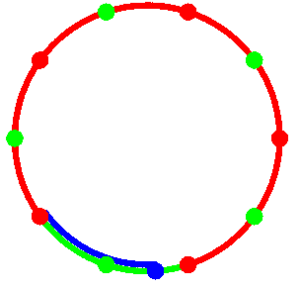}
\fi
	\caption{Situation for $s+t = 5$ and $\arg(q) = 0$. For a fixed $\arg(q)$, the set $\wh U_j^A$ is the union of $s+t$ open segments between the red (dark) points (one is exemplarily depicted in green (light) color here). Each of the segments is retracted to their green (light) midpoint under $\wh F$. For a point $\arg(p)$ (the blue (very dark) point here), the corresponding value $\arg(p) \mod \frac{2\pi}{s+t}$ is the length of the blue (very dark) segment. Thus, indeed, $\wh F(\wh U_j^A,1) = \rho_{(\pi/(s+t),0)}$.}
	\label{Fig:HomotopyTrinomials}
\end{figure}

Now we have all tools to prove the first main theorem of this section, which describes the topology of the sets $U_j^A$ for all $j \neq t$ and their complements.

\begin{thm}
Let $A = \{0,t,s+t\}$. For each $j \in \{1,\ldots,s+t-1\} \setminus \{t\}$ both $U_j^A \subseteq T_A$ and $(U_j^A)^c \subseteq T_A$ are isotopic to the torus knot $K(s+t,s)$. Hence, $U_j^A$ and $(U_j^A)^c$ are connected but not simply connected and we have $\pi_1(U_j^A) = \pi_1((U_j^A)^c) = \Z$.
\label{Thm:TopologyTrinomials}
\end{thm}

\begin{proof}
By Lemma \ref{Lemma:Deformation1Ualp} and \ref{Lemma:Deformation2Ualp}, $U_j^A$ and $(U_j^A)^c$ can be deformation retracted to the closed paths $\rho_{(0,0)}$ and $\rho_{(\pi/(s+t),0)}$ on the standard torus $T_{(1,1)}$. By Corollary \ref{Cor:TorusKnot} both $\rho_{(0,0)}$ and $\rho_{(\pi/(s+t),0)}$ are homeomorphic to $K(s+t,s)$. Since a torus knot $K(s+t,s)$ is an embedding $S^1 \to S^3$
we have in particular $\pi_1(K(s+t,s)) = \Z$ (see \cite{Burde:Zieschang,Hatcher} for further details) and the statement follows.
\end{proof}

Note that for $s=1$ the torus knot $K(s+t,s)$ is a trivial knot. Therefore, e.g., the cubics $x^3+px^2+q$ and $x^3+px+q$ result topologically in non homotopic sets $U_1^A$ and $U_2^A$ although from an algebraic point of view this difference would not be expected a priori. Namely, $U_1^A$ of $x^3+px^2+q$ is isotopic to the trivial knot and $U_2^A$ of $x^3+px+q$ is isotopic to the trefoil knot.

\begin{figure}[ht]
\ifpictures
	\includegraphics[width=0.525\linewidth]{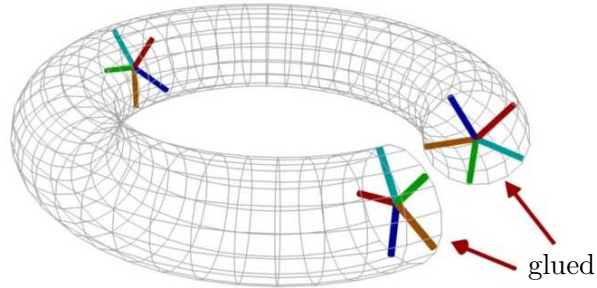}
	\put(-22,16){\small{\mbox{glued}}}
\fi
	\caption{The set $\{f = z^5 + p z^3 + e^{i \cdot \arg(q)} \ : \ p \in \C, |p| \leq 1, \arg(q) \in [0,2\pi)\}$ in the corresponding subset (a real full torus) of its parameter space (considered as $(p,q) \in \C \times \C$). Note that we need restrict to $|p| > 0$ if we want to investigate sets in $T_A$.}
	\label{Fig:TrinomialTopology}
\end{figure}

Finally, we describe the topology of $U_t^A$, its complement and the topology of the discriminant for $A = \{0,t,s+t\}$. We need the following well-known fact; see for example \cite[Exercise 13, Page 39]{Hatcher}.

\begin{lemma}
Let $X$ be a topological space with a path connected subspace $A$ such that $x_0 \in A$. Then the map $\pi_1(A,x_0) \to \pi_1(X,x_0)$ induced by the inclusion $A \hookrightarrow X$ is surjective if and only if every path in $X$ with endpoints in $A$ is homotopic to a path in $A$.
\label{Lem:FundamentalGroupSubspace}
\end{lemma}

\begin{thm}
The zero set $\cV(D) \subseteq T_A$ of the discriminant $D$ is a deformation retract of the set $(U_t^A)^c$ and Theorem     \ref{Thm:TopologyTrinomials} literally also holds for the set $\cV(D)$. For $U_t^A$ we have $\pi_1(U_t^A) = \Z^2$.
\label{Thm:TopologyMiddleTerm}
\end{thm}

\begin{proof}

Let $f = z^{s+t} + p z^{t} + q$ and without loss of generality $s + t$ odd (for $s + t$ even the proof works analogously). First, we deal with $(U_t^A)^c$. By definition $f \in (U_t^A)^c$ only if $f$ is nowhere lopsided with dominating term $p$. This is reflected in the following way. We define for every $q \in \C^*$ the closed punctured disk
\begin{eqnarray*}
	\cB_{q}^{\bullet} & = & \{p \in \C^* \ : \  |p| \leq |q|^{s/(s+t)} \left((t/s)^{s/(s+t)} + (s/t)^{t/(s+t)}\right)\} 
\end{eqnarray*}
in the $\C^*$-slice of $T_A$ given by fixing $q \in \C^*$. By Corollary \ref{Cor:LocalStructureUzero} we know that for every fixed $q \in \C^*$ we have $f \in (U_t^A)^c$ if and only if $p \in F(s,t,q)^{\odd} \cap \cB_{q}^{\bullet}$, which is an
arrangement of $s+t$ half open segments in $\C^*$, and $f \in \cV(D)$ if and only if $p \in (F(s,t,q)^{\odd} \cap \partial \cB_{q}^{\bullet}) \setminus \{0\}$. Thus, we can deformation retract $(U_t^A)^c$ to $\cV(D)$ via the homotopy
\begin{eqnarray*}
	F_1: T_A \times [0,1]  & \to & T_A, \\
		  ((p,q),l) & \mapsto & \left(\frac{p}{(1-l) + l |p| \cdot \left(|q|^{s/(s+t)} \left((t/s)^{s/(s+t)} + (s/t)^{t/(s+t)}\right)\right)^{-1}}, q\right),
\end{eqnarray*}
i.e., we retract every half-open ray segment in $F^{\odd}(s,t,q) \cap \cB_{q}^{\bullet}$ to its intersection point with $\partial \cB_{q}^{\bullet}$.

Recall the definition of the homotopy $F$ in \eqref{Equ:Homotopy} in Lemma \ref{Lemma:Deformation1Ualp}. Since for a fixed $q$ the zero set $\cV(D)$ intersects every half ray of $F(s,t,q)^{\odd}$ in exactly one point (Corollary \ref{Cor:LocalStructureUzero}) and $F$ maps every 
ray of $F(s,t,q)^{\odd}$ to exactly one point located on the closed path $\rho(\pi/(s+t),0)$ (see Lemma \ref{Lemma:Deformation1Ualp}, Definition \eqref{Equ:PathRho} and Lemma \ref{Lemma:Deformation2Ualp}) on the standard 2-torus $T_{(1,1)} \subseteq T_A$, $F_{|\cV(D)}$ deformation retracts $\cV(D)$ to $\rho(\pi/(s+t),0)$ and is a homeomorphism on the subspace of $T_A$ given by fixing $|q|$. Now, the statement follows with Theorem \ref{Thm:TopologyTrinomials}.

Finally, we compute the fundamental group $\pi_1(U_t^A)$. Let now $T_{(|p^*|,1)} \subseteq T_A$ be the torus given by $|p^*| = (t/s)^{s/(s+t)} + (s/t)^{t/(s+t)} + 1$ (and $|q| = 1$). By Theorem \ref{Thm:LocalStructureUalpha} we have $T_{(|p^*|,1)} \subseteq U_t^A$ since every trinomial in $T_{(|p^*|,1)}$ is lopsided with dominating term $p z^t$. Let $x_0$ be the origin in $T_{(|p^*|,1)}$. We investigate the following inclusions.
\begin{eqnarray*}
 (S^1)^2 \simeq T_{(|p^*|,1)} \ \hookrightarrow \ U_t^A \ \hookrightarrow \ T_A \simeq (\C^*)^2.
\end{eqnarray*}

Let $\gamma$ be an arbitrary closed path in $U_t^A$ with start- and endpoint in $T_{(|p^*|,1)}$. Since $q$ is the constant term of every trinomial in $T_A$, we can, by continuously rescaling the norms of the roots, first retract $\gamma$ to a path $\gamma'$, which is contained in the subspace of $T_A$ given by $|q| = 1$. Since for every point $(p,q) \in \gamma \subseteq U_t^A$ and every $\lam > 1$,
Theorem~\ref{Thm:LocalStructureUalpha} implies $(\lam p, q) \in U_t^A$, $\gamma'$ is homotopy equivalent to a path $\gamma'' \in T_{(|p^*|,1)}$. 

An analogous statement holds for an arbitrary path $\gamma$ in $T_A$ with start- and endpoints in $U_t^A$, since we can simply retract $\gamma$ to a path in $T_{(|p^*|,1)} \subseteq U_t^A$. 

Since every point in $U_t^A$ is path connected to $T_{(|p^*|,1)} \subseteq U_t^A$ by the upper argumentation, $U_t^A$ is path connected (alternatively, this fact can also be derived from Corollary \ref{Cor:LocalStructureUzero}). Thus, we can apply Lemma \ref{Lem:FundamentalGroupSubspace} and obtain surjective maps
\begin{eqnarray*}
 \pi_1(T_{(|p^*|,1)},x_0) \ \twoheadrightarrow \ \pi_1(U_t^A,x_0) \ \twoheadrightarrow \ \pi_1(T_A,x_0).
\end{eqnarray*}
Since we know $\pi_1(T_{(|p^*|,1)},x_0) = \pi_1(T_A,x_0) = \Z^2$, we can conclude $\pi_1(U_t^A,x_0) = \Z^2$.
\end{proof}

Note that the statements about $U_j^A$ remain true in $\tilde{T}_A = T_A \cup \{z^{s+t} + q \, : \, q \in \C^*\}$ since $U_j^A \cap \{z^{s+t} + q \, : \, q \in \C^*\} = \emptyset$ for all $1 \leq j \leq s+t-1$. Similarly $\pi_1((U_t^A)^c)$ still equals $\Z$ since it can easily be deformation retracted to $\{z^{s+t} + q \, : \, q \in \C^*\}$ with Theorem \ref{Thm:LocalStructureUalpha}. However, the topology of $(U_j^A)^c$ for $j \neq t$ might be different in $\tilde{T}_A$ since the homotopy $F$ in Lemma \ref{Lemma:Deformation1Ualp} cannot be extended to $\tilde{T}_A$. We do not discuss this last case in further detail.

Instead, we close the article with some remarks about the zero set $\cV(D)$ of the discriminant $D$ of trinomials and its amoeba and coamoeba.

\begin{cor}
Let $D$ be the discriminant of all trinomials with support set $A = \{0,t,s+t\}$ with zero set $\cV(D) \subseteq T_A$. Then the amoeba $\cA(D)$ is a line given by
\begin{eqnarray*}
 \log|p| \ = \ (s/(s+t)) \cdot \log|q| + \log|(t/s)^{s/(s+t)} + (s/t)^{t/(s+t)}|,
\end{eqnarray*}
and the coamoeba $\coA(D)$ is isotopic to the torus knot $K(s+t,s)$.
\label{Cor:DiscriminantAmoeba}
\end{cor}

\begin{proof}
The amoeba statement follows immediately from Corollaries~\ref{Cor:Discriminant}
and~\ref{Cor:LocalStructureUzero}.
For the coamoeba statement recall that $\cV(D) \subseteq (U_t^A)^c$ and $(U_t^A)^c$ deformation retracts to $K(s+t,s)$ by Theorem \ref{Thm:TopologyMiddleTerm}. Since by Lemma \ref{Lemma:Deformation1Ualp} the deformation retraction is given by retracting every fiber of the $\Arg$-map to its intersection point with the unit torus, we can conclude that $\Arg(\cV(D))$ is isotopic to $K(s+t,s)$.
\end{proof}

We remark that the statement about the amoeba $\cA(D)$ is exactly what we would expect a priori from amoeba theory. Due to a result by Passare, Sadykov and Tsikh \cite[Corollary 8]{Passare:Sadykov:Tsikh}, amoebas of principal $A$-determinants are solid, i.e., every component of the amoeba complement corresponds to a vertex in the Newton polytope via the order map. This implies in particular that amoebas of discriminants of univariate polynomials are solid; see \cite[Corollary 9]{Passare:Sadykov:Tsikh}. Since in our case the discriminant $D$ is a bivariate binomial it follows that the complement of $\cA(D)$ has exactly two components. Since each of these components is convex \cite{Gelfand:Kapranov:Zelevinsky}, $\cA(D)$ has to be a line.
 \smallskip

Furthermore, the Theorems \ref{Thm:TopologyTrinomials} and \ref{Thm:TopologyMiddleTerm} are a generalization of the well-known Milnor-fibration for the case of discriminants of trinomials. Recall that Milnor's fibration theorem states that for every $(n+1)$-variate complex polynomial $f$ with a singular point in the origin and every sufficient small $\eps > 0$ we have that
\begin{eqnarray}
 \frac{f}{|f|}: (S^{2n+1}_\eps \setminus \cV(f)) \to S^1 \label{Equ:MilnorFiberTheorem}
\end{eqnarray}
is a fibration. Here $S^{k}_\eps$ denotes the sphere of real dimension $k$ around the origin with radius $\eps$ and $\cV(f)$ denotes the zero set of $f$. Each fiber of the fibration is a smooth parallelizable manifold of real dimension $2n$. The boundary of this fiber corresponds to the intersection $S^{2n+1}_\eps \cap \cV(f)$ and is a compact manifold of real dimension $2n - 1$, which is called \textit{Milnor fiber}. In the special case of $n = 1$ the Milnor fiber is a fibered knot. In general, this fibration does not extend to arbitrary radii of the sphere. For further details see e.g. \cite{Dimca:Book,Milnor:Book}.

If we embed the space $T_A$ in the parameter space $\C^2$ (i.e., we allow $p,q = 0$), then Milnor's fiber theorem states that the zero set $\cV(D)$ of the discriminant $D$ intersected with a 3-sphere of small radius $\eps$ around the origin is diffeomorphic to a fibered knot. Our Theorems \ref{Thm:TopologyTrinomials} and \ref{Thm:TopologyMiddleTerm} show that this knot, the Milnor fiber, is the $K(s+t,s)$ torus knot and in particular that this diffeomorphism extends to the whole space $T_A$. Thus, the fibration \eqref{Equ:MilnorFiberTheorem} extends to the whole space and hence is a fibration $T_A \setminus \cV(D) \to S^1 \times \R_{> 0}$ in our case.

Our results about the discriminant furthermore reprove and generalize a large part of a Theorem by Libgober \cite[Theorem B]{Libgober}. He had shown that for trinomials with $t = 1$ the space $T_A \setminus \cV(D)$ has the fundamental group given by the torus knot $K(s+1,s)$.
Libgober showed additionally that (for arbitrary $t$) the space $T_A \setminus \cV(D)$ is the Eilenberg MacLane space for $s+t$ and $t$.

\section{Final Remarks}
\label{Sec:Openquestions}
In this paper, we focussed on trinomials and have studied the geometry and topology 
of the space of trinomials with respect to the norms of their roots. 
Beyond trinomials, studying the
geometry of polynomials is a well-established
field (see, e.g., Marden's book \cite{Marden:Book}), which in the last years has seen a lot of 
renewed interest, notably through the work of Br\"and\'{e}n and Borcea (see e.g., \cite{Borcea:Braenden:ApplicationsStablePolynomials,Borcea:Braenden:PolyaSchur} and the 
survey \cite{Wagner:StablePolynomialsSurvey}).
Extending the trinomial situation from the current paper to general polynomials leads to
the following question:
Given a fixed support $A \subseteq \{0, \ldots, n\}$, what is the geometry and the topology
of the space of univariate polynomials of degree $n$ with support $A$, 
with regard to the norms of their roots.
In \cite{Vassiliev:TopologyDiscriminants},
Vassiliev has studied for general polynomials the topology of discriminants 
and their complements. Topologically, this question has a long history. In 1970 Arnold proved in \cite{Arnold} that for an arbitrary parameter space $\C^A$ with $A = \{0,\ldots,d\} \subseteq \Z$ with discriminant $D$ the space $\C^A \setminus \cV(D)$ is diffeomorphic to the space $\R^2(d)$ of all subsets of $\R^2$ with cardinality $d$. Particularly, the fundamental group of $\R^2(d)$ is the $d$-th braid group. Braid groups go back to Artin in the 1920's; see \cite{Artin}. Complements of discriminants have various applications in and are connected to different branches of mathematics, e.g. Smale's topological complexity of algorithms \cite{Smale}. See Vassiliev's book \cite{Vassiliev:Book} for an overview and further details.

However, studying our trinomial questions in the more general context of polynomials with a fixed support would correspond to study a ``norm discriminant'' for polynomials which has yet to be developed.

\bibliographystyle{amsplain}
\bibliography{./AmoebasUnivarCase}

\end{document}